\documentclass[10pt]{article}
\usepackage{amssymb,amsmath}
\usepackage[utf8]{inputenc}
\usepackage[english]{babel}
\usepackage{amsmath, amssymb, amsfonts, amsthm,a4wide}
\usepackage{graphicx}
\usepackage{color}
\usepackage{bm}
\usepackage{subfigure}
\usepackage{graphicx}
\usepackage{mathabx}
\usepackage{multirow}
\usepackage{setspace}
\usepackage{titlesec}

\newtheorem{theorem}{Theorem}[section]
\newtheorem{lemma}[theorem]{Lemma}
\newtheorem{proposition}[theorem]{Proposition}
\newtheorem{corollary}[theorem]{Corollary}

\newtheorem{definition}{Definition}[section]
\newtheorem{remark}{Remark}[section]
\newtheorem{example}{Example}[section]

\titleformat{\paragraph}
{\normalfont\normalsize}{\theparagraph}{1em}{}
\titlespacing*{\paragraph}
{0pt}{3.25ex plus 1ex minus .2ex}{1.5ex plus .2ex}

\DeclareFontFamily{U}{MnSymbolC}{}
\DeclareSymbolFont{MnSyC}{U}{MnSymbolC}{m}{n}
\DeclareFontShape{U}{MnSymbolC}{m}{n}{
    <-6>  MnSymbolC5
   <6-7>  MnSymbolC6
   <7-8>  MnSymbolC7
   <8-9>  MnSymbolC8
   <9-10> MnSymbolC9
  <10-12> MnSymbolC10
  <12->   MnSymbolC12}{}
\DeclareMathSymbol{\lefthook}{\mathbin}{MnSyC}{'270}

\title{Complex Engel Structures}
\footnotesize{\author{Zhiyong Zhao\\
Mathematics Department\\
Duke University, Durham, NC 27708-0230, USA\\
E-mail:  zhiyong.zhao@duke.edu}
}

\date{}
\begin{document}
\maketitle
\begin{center}{\bf Abstract} \end{center}
\begin{center}
\begin{minipage}{130mm}  
We study the geometry of Engel structures, which are 2-plane fields on 4-manifolds satisfying a generic condition, that are compatible with other geometric structures. A complex Engel structure is an Engel 2-plane field on a complex surface for which the 2-planes are complex lines. We solve the equivalence problems for complex Engel structures and use the resulting structure equations to classify homogeneous complex Engel structures. This allows us to determine all compact, homogeneous examples. Compact manifolds that support homogeneous complex Engel structures are  diffeomorphic to $S^1\times SU(2)$ or quotients of  $\mathbb{C}^2$, $S^1\times SU(2)$, $S^1\times G$ or $H$ by co-compact lattices, where $G$ is the connected and simply-connected Lie group with Lie algebra $\mathfrak{sl}_2(\mathbb{R})$ and $H$ is a solvable Lie group. 



\end{minipage}
\end{center}

\section{Introduction}
A \emph{distribution} is a subbundle $D\subset TM$ of the tangent bundle of a manifold $M$. We will consider certain distributions with special properties, for example, distributions with some integrability conditions.  Given a distribution of rank $n$ on a manifold $M^m$, if, for each point of $M$, there exists a coordinate neighborhood $U$ and local coordinates $x_1,x_2,\cdots,x_m$ such that $\tfrac{\partial}{\partial x_i}$, $i = 1,\cdots,n$ forms a local basis for the distribution on $U$, then the distribution is said to be \emph{completely integrable}. Besides complete integrability, one can consider partially integrable distributions. An extreme condition is to be nowhere integrable, i.e., for every $p\in M$, there exist $X, Y$, sections of $D$, such that $[X, Y]_p$ is not a section of $D$. For example, contact structures are nowhere integrable distributions on odd dimensional manifolds. The study of contact structures \cite{MR2194671} usually involves interplay between geometry, topology and dynamics. Contact structures play an important role in the study of low-dimensional topology.

 We will study \emph{Engel structures}, which are certain non-integrable distributions defined on 4-manifolds. We will see that, locally, all Engel structures are isomorphic but the global theory of Engel structures is not trivial. A 4-manifold can carry many nonisomorphic Engel structures \cite{MR1350504}. There are relations between contact structures on 3-dimensional manfiolds and Engel structures. For example, V. Gershkovich \cite{MR1350504} proved that each Engel manifold carries a canonical one-dimensional foliation and an Engel structure defines a contact distribution on any three-dimensional submanifold transversal to the canonical foliation.

We will solve the equivalence problem for complex Engel structures and present the classification of compact quotients of homogeneous complex Engel structures. Engel structures (to be defined below) can be characterized in terms of the derived system construction.
 
 \begin{proposition} \cite{MR1083148}\label{pro:1}
 Given a Pfaffian system $I$, there exists a bundle map $\delta : I\rightarrow\Lambda^2(T^*M/I)$ that satisfies $\delta\omega \equiv d\omega\mod(I)$ for all $\omega\in\Gamma(I)$.
 \end{proposition}

 \begin{definition}
 By Proposition \ref{pro:1}, we have a bundle map $\delta$. Set $I^{(1)} = \ker \delta$ and call $I^{(1)}$ \emph{the first derived system}. 
  Continuing with this construction, we can get a filtration
  \[I^{(k)}\subset \cdots \subset I^{(2)}\subset I^{(1)}\subset I^{(0)} = I,    \]
  defined inductively by
  \[I^{(k+1)} = (I^{(k)})^{(1)}\, .\]
$I^{(k)}$ is called \emph{the $kth$ derived system}. 
 \end{definition}
 Now we present the definition and characterization of Engel structures.
 
\begin{definition}[Engel Structure]
Given a 4-manifold $M$ and a Pfaffian system $I\subset T^*M$, an \emph{Engel structure} is a sub-bundle $D = I^\perp$ of the tangent bundle of $M$ that satisfies: $(1)$ $I$ is of rank 2, $2$  $I^{(1)}$ is of rank 1 and $3$ $I^{(2)} = 0$.  A manifold endowed with an Engel structure $D = I^\perp$ is called an \emph{Engel manifold}. 
\end{definition} 

\begin{definition}
Given an ideal $\mathcal I$ generated by a Pffafian system $I$, a vector field $\xi$ is called a \emph{Cauchy characteristic vector field} of  $\mathcal I$ if $\xi\lefthook \mathcal I \subset \mathcal I$. At a point $x\in M$, the set of Cauchy characteristic vector fields is
\[A(I)_x = \{\xi_x\in T_xM | \xi_x\lefthook\mathcal I_x \subset\mathcal I_x\} \subset I^\bot\]
and the \emph{retracting space} or \emph{Cartan system} is defined to be
\[C(I)_x = A(I)_x^\bot \subset T_x^*M\, .\]
\end{definition}
By the definition of Engel structure $I^\bot$, there is a canonical flag of sub-bundles
\[0\subset I^{(1)}\subset I\subset C(I^{(1)}) \subset T^*M\, .\]
V. Gershkovich \cite{MR1350504} proved the following theorem which can also be found in \cite{MR1452539}.
\begin{theorem}
If an orientable 4-manifold admits an orientable Engel structure, then it has trivial tangent bundle.
\end{theorem}

T. Vogel \cite{MR2480602} proved the converse of the above theorem:
 \begin{theorem}
Every parallelizable 4-manifold admits an orientable Engel structure.
\end{theorem}
Thus for an orientable 4-manifold, parallelizability is equivalent to the existence of an orientable Engel structure. This is a global characterization of manifolds that support orientable Engel structures. Locally, we have the following \emph{Engel normal form} \cite{MR1083148}, which implies that there is no local invariant for Engel structures, i.e., all Engel structures are locally equivalent.
\begin{theorem}[Engel normal form] \label{the:1}
Let $I$ be a Pfaffian system on $M^4$ such that $I^\perp\subset TM$ is an Engel structure. Then every point of $M$ has an open neighborhood $U$ on which there exists local coordinates  $(x, y_0, y_1, y_2): U\rightarrow \mathbb{R}^4$ such that
 \[I|_{U} =\{dy_0-y_1 dx, dy_1 - y_2 dx\}\, . \]
\end{theorem}
In this paper, we will consider complex Engel structures.
\begin{definition}
Given a complex manifold $(M, J, I)$ with an Engel structure $D = I^\perp$, if $I^\bot\subset TM$ is a complex line field, the Engel structure is called a \emph{complex Engel structure}.
\end{definition}
 
\section{Geometry of Complex Engel Structures}

\noindent Let $J$ be the complex structure on the underlying manifold $M$. Choose a  local $J$-complex coframing $(\omega_1, \omega_2)$ on an open set $U\subset M$ such that
\begin{enumerate}
\item $\omega_1, \omega_2$ are of $J$-type (1,0)
\item $\omega_2 = 0$ defines $I^\perp$ on $U$
\end{enumerate}
\noindent Then $(\omega_1, \omega_2)$ is called a \emph{0-adapted coframing} for the Engel structure.
\begin{lemma}
The $0$-adapted coframings are the sections of a $G$-structure on $M$, where $G\subset GL(2, \mathbb{C})$ is the 3-dimensional complex subgroup
\[G = \left\{\left.\left(\begin{array}{cc}
a & b\\
0& c
\end{array}
\right)\right| a,b,c\in\mathbb{C} \text{ and } a, c\neq 0\right\}\, .\]
\end{lemma}

In the following analysis, we will denote the conjugate of $\omega_1, \omega_2$ by $\bar\omega_1, \bar\omega_2$ instead of $\overline{\omega_1}, \overline{\omega_2}$.

\begin{theorem}\label{theo:coframing}
A complex Engel structure has a canonical coframing $(\omega_1,\ \omega_2)$ (i.e., an $e$-structure) such that
\begin{equation}
\begin{aligned}
d\omega_2 &\equiv \omega_1\wedge\bar{\omega}_1&\ \ \ \ \  &\mod\ \ \  &\omega_2\, ,\\
d(\omega_2 + \bar{\omega}_2) &\equiv -\frac{1}{2}(\omega_1 -\bar{\omega}_1)\wedge (\omega_2 - \bar{\omega}_2)&\ \ \ \ \  &\mod\ \ \  &\omega_2 + \bar{\omega}_2\, .
\end{aligned}
\end{equation}

\end{theorem}

\begin{proof}
Since $\omega_2 = 0$ defines the complex Engel structure, a 0-adapted coframing $  (\omega_1, \omega_2)$ is defined on an open set $U\subset M$ up to the following change of coframing:
\[\left(\begin{array}{c}
\hat{\omega}_1\\
\hat{\omega}_2
\end{array}
\right) = \left(\begin{array}{cc}
a & b\\
0& c
\end{array}
\right)
\left(\begin{array}{c}
\omega_1\\
\omega_2

\end{array}
\right),\]
where $a, b, c$ are complex functions on $U$ and $ac \neq 0$ on $U$.

\noindent The Engel condition implies that
\[d\omega_2 \equiv A\, \omega_1\wedge\bar{\omega}_1\ \ \ \ \  \mod\ \ \  \omega_2, \bar{\omega}_2\, ,\]
where $A \neq 0$. Define $\hat\omega_2 = \frac{1}{A}\omega_2$, then
\begin{equation}
\begin{aligned}
d\hat\omega_2 &= d\left(\tfrac{1}{A}\right) \wedge\omega_2 + \tfrac{1}{A} d\omega_2\\
 			 &\equiv \tfrac{1}{A} A\omega_1\wedge\bar{\omega}_1\ \ \ \ \  \mod\ \ \  \omega_2, \bar{\omega}_2\\
			 &\equiv \omega_1\wedge\bar{\omega}_1\ \ \ \ \  \mod\ \ \  \omega_2, \bar{\omega}_2\, .
\end{aligned}
\end{equation}
Thus, after rescaling $\omega_2$, we can arrange that $A = 1$. Then
\begin{equation}\label{eq:83}
d\omega_2 \equiv \omega_1\wedge\bar{\omega}_1\ \ \ \ \  \mod\ \ \  \omega_2, \bar{\omega}_2\, .
\end{equation}
Such coframings will be said to be $1$-adapted. They are the sections of a $G_1$-structure, where $G_1\subset G$ is defined by $c=a\bar{a}$.
The change of the coframing that preserves (\ref{eq:83}) is reduced to
\begin{equation}\label{eq:84}
\left(\begin{array}{c}
\hat{\omega}_1\\
\hat{\omega}_2
\end{array}
\right) = \left(\begin{array}{cc}
a & b\\
0& a\bar{a}
\end{array}
\right)
\left(\begin{array}{c}
\omega_1\\
\omega_2
\end{array}
\right)\, .
\end{equation}
Now $\omega_2$ is defined up to a real multiple, so the real and imaginary parts of $\omega_2$ are uniquely defined up to a real multiple. Suppose the real part of $\omega_2$ spans $I^{(1)}$, the first derived system.

By (\ref{eq:83}), the real part of $\omega_2$ spans $I^{(1)}$, the first derived system, i.e.,
\[d(\omega_2 + \bar{\omega}_2) \equiv 0\ \ \ \ \  \mod\ \ \  \omega_2, \bar{\omega}_2\, .\]
Since $d(\omega_2 + \bar{\omega}_2)$ is a real 2-form, there must be a complex function $p_1$ such that
\[
d(\omega_2 + \bar{\omega}_2) \equiv (p_1\omega_1 - \bar{p}_1\bar{\omega}_1)\wedge (\omega_2 - \bar{\omega}_2)\ \ \ \ \  \mod\ \ \  \omega_2 + \bar{\omega}_2\, .
\]
Because $I^{(2)} = (0)$, $d(\omega_2 + \bar{\omega}_2)\wedge(\omega_2 + \bar{\omega}_2) \not\equiv 0$ implies that $p_1\neq 0$. By replacing $(\omega_1, \omega_2)$ by $\left(-\tfrac{1}{2p_1}\omega_1, \tfrac{1}{4p_1\bar{p}_1}\omega_2\right)$, we can arrange $p_1 = -\tfrac{1}{2}$. After this arrangement, $a$ is fixed to be 1 in the transformation. Now the coframing satisfies
\begin{equation}\label{eq:1}
d(\omega_2 + \bar{\omega}_2) \equiv -\tfrac{1}{2}(\omega_1 -\bar{\omega}_1)\wedge (\omega_2 - \bar{\omega}_2)\ \ \ \ \  \mod\ \ \  \omega_2 + \bar{\omega}_2\, .
\end{equation}
Such coframings will be said to be $2$-adapted. They are the sections of a $G_2$-structure, where $G_2\subset G_1$ is defined by $a = 1$. After setting $a = 1$, by (\ref{eq:84}), $\omega_2$ is unique and $\omega_1$ is unique modulo $\omega_2$, i.e. a change of coframing that preserves (\ref{eq:83}) and (\ref{eq:1}) is reduced to 
\[
\left(\begin{array}{c}
\hat{\omega}_1\\
\hat{\omega}_2
\end{array}
\right) = \left(\begin{array}{cc}
1 & b\\
0& 1
\end{array}
\right)
\left(\begin{array}{c}
\omega_1\\
\omega_2

\end{array}
\right)\, .
\]

\noindent Because $\omega_2$ is uniquely defined now, we can write 
\[d\omega_2 \equiv \omega_1\wedge\bar{\omega}_1 + f\omega_1\wedge \bar{\omega}_2\ \ \ \ \  \mod\ \ \  \omega_2\]
for some function $f$. Note that there is no $\bar{\omega}_1\wedge\bar{\omega}_2$ term since we assumed that $(\omega_1, \omega_2)$ be of type (1,0), and the underlying almost complex structure is integrable.

By adding a multiple of $\omega_2$ to $\omega_1$, we can arrange $f = 0$. 
\begin{equation}\label{eq:27}
d\omega_2 \equiv \omega_1\wedge\bar{\omega}_1\ \ \ \ \  \mod\ \ \  \omega_2.
\end{equation}
After arranging this modification, the coframing satisfying (\ref{eq:1}) and (\ref{eq:27}) is completely determined: the original $G$-structure defines a canonical sub e-structure.
\end{proof}
Thus by a Theorem of Kobayashi \cite{MR0088766},
\begin{corollary}
The symmetry group of a complex Engel structure acts freely on the underlying connected manifold.
\end{corollary}

\begin{theorem}
The canonical coframing of a complex Engel structure satisfies
\begin{equation}\label{eq:str}
\begin{aligned}
d\omega_1& = -(p_1\omega_1 + p_2\omega_2 + \bar{q}_1\bar{\omega}_1 + \bar{q}_2\bar{\omega}_2)\wedge\omega_1 - (q_2\omega_1 + \bar{r}_1\bar{\omega}_1 + \bar{r}_2\bar{\omega}_2)\wedge\omega_2\\
d\omega_2& = (\omega_1 - \omega_2)\wedge\bar{\omega}_1  - (p_1\omega_1 + p_2\omega_2 + \bar{p}_1\bar{\omega}_1 + \bar{p}_2\bar{\omega}_2)\wedge\omega_2
\end{aligned}
\end{equation}

\noindent where $p_1, p_2, q_1, q_2, r_1, r_2$ are complex functions. Thus, a complex Engel structure has $6$ fundamental functional invariants.
\end{theorem}

\begin{proof}
From Theorem \ref{theo:coframing}, the structure equation can be writen as 
\[d\omega_2 = \omega_1\wedge\bar{\omega}_1 + (p\omega_1 + q\bar{\omega}_1 + r\bar{\omega}_2)\wedge\omega_2\]
for some functions $p, q, r$. Therefore 
\begin{equation}\label{eq:2}
\begin{split}
d(\omega_2 + \bar{\omega}_2) &\equiv (p\omega_1 + q\bar{\omega}_1 + r\bar{\omega}_2)\wedge\omega_2 +
(\bar{p}\bar{\omega}_1 + \bar{q}\omega_1 + \bar{r}\omega_2)\wedge\bar{\omega}_2\\
&\equiv ((p - \bar{q})\omega_1 + (q - \bar{p})\bar{\omega}_1)\wedge\omega_2\ \ \ \ \  \mod\ \ \  (\omega_2+\bar{\omega}_2)\, .
\end{split}
\end{equation}
But according to (\ref{eq:1})
\begin{equation}\label{eq:3}
\begin{split}
d(\omega_2 + \bar{\omega}_2)& \equiv -\tfrac{1}{2}(\omega_1 -\bar{\omega}_1)\wedge (\omega_2 - \bar{\omega}_2)\\
& \equiv -(\omega_1 -\bar{\omega}_1)\wedge \omega_2
\ \ \ \ \  \mod\ \ \  (\omega_2 + \bar{\omega}_2)\, .
\end{split}
\end{equation}
By comparing (\ref{eq:2}) and (\ref{eq:3}), we find $q = \bar{p} + 1$. Thus
\[d\omega_2 = \omega_1\wedge\bar{\omega}_1 +\bar{\omega}_1\wedge\omega_2+ (p\omega_1 + \bar{p}\bar{\omega}_1 + r\bar{\omega}_2)\wedge\omega_2\, .\]
Let $\alpha = -(p\omega_1 + \bar{r}\omega_2)$, then
\begin{equation}\label{eq:28}
d\omega_2 = \omega_1\wedge\bar{\omega}_1 +\bar{\omega}_1\wedge\omega_2 - (\alpha + \bar{\alpha})\wedge\omega_2\, ,
\end{equation}
where $\alpha$ is a $(1,0)$-form, uniquely defined by (\ref{eq:28}). Let $\gamma$ be a $(1,0)-$form and $\beta$ be any 1-form. The structure equation can be written as
\begin{equation}\label{eq:4}
\begin{aligned}
d\omega_1& = -(\alpha + \bar{\gamma})\wedge\omega_1 - \beta\wedge\omega_2\, ,\\
d\omega_2& = \omega_1\wedge\bar{\omega}_1 +\bar{\omega}_1\wedge\omega_2 - (\alpha + \bar{\alpha})\wedge\omega_2\, .
\end{aligned}
\end{equation}

\noindent Taking the exterior derivative of $d\omega_2$ then yields
\[-\bar{\gamma}\wedge\omega_1\wedge\bar{\omega}_1 + \omega_1\wedge\bar{\beta}\wedge\bar{\omega}_2 + \omega_1\wedge \gamma\wedge\bar{\omega}_1 \equiv0\ \ \mod \omega_2\, .\]
Recall that $\gamma$ is a $(1,0)$-form, so
\[
\bar{\gamma}\wedge\bar{\omega}_1 + \bar{\beta} \wedge\bar{\omega}_2 \equiv 0\ \ \ \ \  \mod\ \ \  \omega_1,\ \omega_2\, .
\]

\noindent Let $\gamma = q_1\omega_1 + q_2\omega_2$, then
$$\beta \equiv q_2\omega_1\ \ \ \ \  \mod\ \ \  \ \bar{\omega}_1,\ \omega_2,\ \bar{\omega}_2\, .$$
Thus, the final structure equation of a complex Engel structure is
\begin{equation}\label{eq:85}
\begin{aligned}
d\omega_1& = -(p_1\omega_1 + p_2\omega_2 + \bar{q}_1\bar{\omega}_1 + \bar{q}_2\bar{\omega}_2)\wedge\omega_1 - (q_2\omega_1 + \bar{r}_1\bar{\omega}_1 + \bar{r}_2\bar{\omega}_2)\wedge\omega_2\, ,\\
d\omega_2& = (\omega_1 - \omega_2)\wedge\bar{\omega}_1  - (p_1\omega_1 + p_2\omega_2 + \bar{p}_1\bar{\omega}_1 + \bar{p}_2\bar{\omega}_2)\wedge\omega_2\, .
\end{aligned}
\end{equation}
\end{proof}

\begin{remark}
We can take exterior derivatives of (\ref{eq:85}), and see that there are no further relations on $p_1, p_2, q_1, q_2, r_1, r_2$. All differential invariants of complex Engel structures are these six or their derivatives with respective to the canonical coframing.
\end{remark}

\section{Homogeneous Complex Engel Structures}

In this section, we will classify homogeneous complex Engel structures. The group of diffeomorphisms preserving a complex Engel structure also preserves its canonical coframing and hence preserves its fundamental invariants. Thus, if it is homogeneous, then the invariants must be constant. Assume that the functions be constant and take exterior derivatives of $d\omega_1$, $d\bar{\omega}_1$, $d\omega_2$ and $d\bar{\omega}_2$, and set all of these to be zero. This will yield quadratic equations on $p,\ q,\ r$. After solving these equations, which was done with the help of MAPLE, we arrive at the following theorem:
 \begin{theorem}[Classification of Homogeneous Complex Engel Structures]\label{th: 11}
There are six distinct two-parameter families of homogeneous structure equations of complex Engel structures. The constants $(p_1, p_2, q_1, q_2, r_1, r_2)$ in equation (\ref{eq:str}) are listed as follows for the six cases:
 \begin{itemize}
 \item Case C1: $(a + ib,\  0,\  0,\  0,\  0,\  0)$
 \item Case C2: $(\frac{1}{2} + ib,\ 0,\ 2ia,\  0,\ 0,\ 0)$
 \item Case C3: $(\frac{1}{2} - ib,\  0,\  2ib,\ \frac{1}{2}(2b + i)(2ia - b),\ 2b^2 - ib,\ (b^2 + \frac{1}{4})(2a + ib))$
 \item Case C4: $(0,\ a-ib,\ 2ib,\ a-ib,\ -a-ib,\ 0)$
 \item Case C5: $(a(1-2ib),\ \frac{1}{4}(2a - 1)(2b + i)^2,\ 2ib,\ -\frac{1}{4}(2b + i)^2,\\ -\frac{1}{4}(2b + 4ab + i(2a - 1))(-2b + i),\ \frac{1}{4}a(-1 + 2bi)(1+ 4b^2) )$
 \item Case C6: $(\frac{1}{2}(\cos a + i \sin a)(2b + i),\ \frac{1}{4}(-\sin a + 2b\cos a  - 1)(2b + i)^2,\ 2ib,\\ -\frac{1}{8}(2i b\sin a  + i\cos a - 2b\cos a  - 2i b + \sin a + 1)(2b + i)^2,\\ \frac{1}{4}(1 + 2bi)(2ib\sin a  + i\cos a + 2b\cos a  - 2ib -\sin a - 1),\\  \frac{1}{16}(1 + 4b^2)(-1 + 2ib)(2ib\sin a  + i\cos a + 2b\cos a  - 2ib -\sin a - 1)   )$
 \end{itemize}
 where $a$ and $b$ are real constants.
\end{theorem}
\begin{proof}
First take the exterior derivatives of $d\omega_2$ and $d\bar{\omega}_2$, which yield
\begin{equation}\label{eq:29}
\begin{aligned}
r_2 &= p_1q_2 + p_2 - q_2\, ,\\
\bar{r}_2 &= \bar{p}_1\bar{q}_2 + \bar{p}_2 -\bar{q}_2\, .
\end{aligned}
\end{equation}
Substituting these into exterior derivatives of $d\omega_2$ and $d\bar{\omega}_2$ yields $q_1 + \bar{q}_1 = 0$
So $q_1$ is pure imaginary. Set 
\begin{equation}\label{eq:86}
q_1 = 2iq_0\, ,\ \bar{q}_1 = -2iq_0\, .
\end{equation}
Substituting these relations into $d\omega_2$ and $d\bar{\omega}_2$ yields
\[\Im{p_2}= q_0(p_1 + \bar{p}_1 - 1)\, ,\]
where $\Im{p_2}$ means the imaginary part of $p_2$.

Now take the exterior derivatives of $d\omega_1$ and $d\bar{\omega}_1$ and set these to be zero. The equations are quadratic expressions in the coefficients of $d\omega_1, d\omega_2$. Then solve these quadratic equations. We get the six different 2-parameter solutions, listed in the Theorem.
\end{proof}

\section{Compact Homogeneous Complex Engel Structures}

We have proved the classification result for homogeneous complex Engel structures. Now we can classify compact homogeneous complex Engel structures. We will prove the following theorem:
\begin{theorem}[Classification of Compact Homogeneous Complex Engel Structures]\label{th:10}
Let $\mathfrak{g}$ be the 4-dimensional Lie algebra of symmetry vector fields of a complex Engel structure. For the six distinct 2-parameter families of homogeneous complex Engel structures listed in Theorem \ref{th: 11}, the results about compactness in each case are listed as follows:
 \begin{itemize}
 \item Case C1:  If $a = \frac{1}{2}$ and $b = 0$, the Lie algebra is a 4-dimensional solvable Lie algebra. There exists a compact quotient that supports a homogeneous complex Engel structure if and only if $a = \frac{1}{2}$ and $b = 0$.
 
 \item Case C2: If there exists a complex number $\lambda$ and a matrix $A\in SL_3(\mathbb{Z})$ such that
\begin{enumerate}
\item $b = -a$
\item  $\left|\lambda\right|\neq 1$
\item the eigenvalues of $A$ are $(\lambda\bar\lambda)^{-1}, \lambda, \bar{\lambda}$
\item  there exists $k\in\mathbb{Z}$ such that $-\tfrac{1}{2a} \log( \left|\lambda\right|) = \arg{\lambda} + 2k\pi$
\end{enumerate}
then there exists a co-compact lattice $\Gamma$ such that $G/\Gamma$ supports a homogeneous complex Engel structure.

 \item Case C3: 
 \begin{enumerate}
 \item  If $a = -\frac{1}{4}$, the Lie algebra is a solvable Lie algebra, and there exists a compact quotient.
 \item If $a = b^2 \neq 0$, the Lie algebra is a solvable Lie algebra, but there does not exist a compact quotient.
\item If $(a<-\frac{1}{4})$ or $(0\leq a < b^2)$ or $(b = 0\ \ \text{and}\ \ a<0)$, the Lie algebra is $\mathbb{R}\times\mathfrak{sl}(2,\mathbb{R})$. There exists a compact quotient.

\item If   $(-\frac{1}{4} < a < 0)$ or $(a > b^2)$ or $(b = 0\ \ \text{and}\ \ \ a >0)$, the Lie algebra is $\mathbb{R}\times\mathfrak{su}(2)$. There exists a compact quotient.

 \end{enumerate}
%

 \item Case C4: There is no compact quotient that supports a homogeneous complex  Engel structure. 
 
\item Case C5: If $a = \tfrac{1}{2}$, there exists a co-compact lattice $\Gamma$ of a solvable  Lie group $G$ that $G/\Gamma$ supports a homogeneous complex Engel structure. 
  
 \item Case C6: There is no compact quotients that supports a  homogeneous complex Engel structure unless $a = -\frac{\pi}{2} + 2k\pi,\ k\in \mathbb{Z}$ and $b = 0$. Under this condition, this is a special case of case $C1$.

 \end{itemize}
 
In summary, compact quotients that support homogeneous complex Engel structures can occur in case $C1$, case $C2$, case $C3$, and case $C5$. 
\end{theorem}
We will prove the theorem in the following section by analyzing the structure equation for each case. If the structure equation is solvable, we will also provide a local coordinate system  expression of the coframing.

\section{Proof of Theorem \ref{th:10}}

\subsection{Homogeneous Case C1}

In this case, $(p_1, p_2, q_1, q_2, r_1, r_2) = (a + ib,\  0,\  0,\  0,\  0,\  0)$. The structure equation is
\begin{equation}\label{eq:case11}
\begin{aligned}
d\omega_1& = 0\, ,\\
d\omega_2& = \left(\omega_1 - \omega_2\right)\wedge\bar{\omega}_1  - \left(\left(a+ ib\right)\omega_1 +  \left(a - ib\right)\bar{\omega}_1 \right)\wedge\omega_2\, .
\end{aligned}
\end{equation}
By (\ref{eq:case11}),
\begin{equation}\label{eq:72}
d(\omega_1\wedge\omega_2\wedge\bar{\omega}_2) = -(1 -2a +2bi)\omega_1\wedge\bar{\omega}_1\wedge\omega_2\wedge\bar{\omega}_2\, .
\end{equation}
By Stokes' Theorem, there is no compact example unless $a = \frac{1}{2}$ and $b = 0$.

By (\ref{eq:case11}), $d\omega_1 = 0$. By the complex Poincar\'e Lemma, there exists a holomorphic function $z$ locally on the manifold such that
\begin{equation}\label{eq:25}
\omega_1 = dz\, .
\end{equation}
\noindent Thus
\[ d\omega_2 = dz\wedge d\bar{z}  +[-(a + ib)dz + (1 - a + ib)d\bar{z}] \wedge \omega_2\, .\]

We will find a local coordinate system for the coframing $(\omega_1, \omega_2)$ in order to explicitly describe its group of symmetries. Since the groups of symmetries  are different for different $a$ and $b$, we will consider two cases: $a + ib = 1$ (special case) and generic case.

\subsubsection{Special Case}

If $a + ib = 1$, i.e. $ a = 1$ and $b = 0$, 
\begin{equation}\label{eq:65}
d(\omega_2 + \bar{z} dz) = -dz\wedge (\omega _ 2 + \bar{z}dz)\, .
\end{equation}
By the complex Frobenius Theorem, there exists a complex function $f$ and a holomorphic function $w$ locally on the manifold such that
\begin{equation}\label{eq:26}
\omega_2 + \bar{z} dz = fdw\, .
\end{equation}
Since $(\omega_1,\ \omega_2)$ is a coframing, $\omega_1$ and $\omega_2$ are linearly independent. By comparing the local coordinate expressions (\ref{eq:25}) and (\ref{eq:26}), we know that  $f$ is nowhere zero on its defining domain. Substituting  (\ref{eq:26}) into (\ref{eq:65}) yields
\[df\wedge dw = -f\ dz\wedge dw\, .\]
Setting $f = e^{-z} g$ for some function $g \neq 0$, we have $dg\wedge dw = 0$. So the function $g$ is a function of $w$ only.
\[\omega_2 + \bar{z} dz = e^{-z} g(w)dw\, .\]
By defining $\tilde{w} = \int g(w)dw$ and dropping the tilde in the local coordinate, we get
\[\omega_2 + \bar{z} dz = e^{-z} dw\, .\]
Since $(\omega_1,\omega_2)$ is a $(1,0)$-coframing, $(z, w)$ can serve as a local holomorphic coordinate system in a neighborhood of the manifold. In the coordinate system of $(z,w)$, the coframing can be expressed as
\begin{equation}\label{eq:loc1}
\begin{aligned}
\omega_1& = dz\, .\\
\omega_2 &= -\bar{z} dz + e^{-z} dw\, .
\end{aligned}
\end{equation}

We consider the symmetry group of the coframing in these local coordinate. Let 
\[\Gamma = \{ k_1 + i k_2| (z,w) \rightarrow(z, w+ k_1 + i k_2),\  \text{where}\ \  k_1, k_2\in \mathbb{Z}\}\, . \]

$\Gamma$ acts freely and discontinuously on the coframing in the local coordinate. So we can take a global model for this complex Engel structure
\[M^4 =\mathbb{C} ^ 2/\Gamma \cong \mathbb{R}^2\times T^2\, .\]

\subsubsection{Generic Case}

If $a + ib \neq 1$, i.e. $a \neq 1$ or $b\neq 0$ , by (\ref{eq:case11}), we have
\begin{equation}\label{eq:10}
d\left(\omega_2 - \frac{dz}{1-a+ib}\right) = \left[-(a + ib)dz + (1 - a + ib)d\bar{z}\right] \wedge \left(\omega_2 - \frac{dz}{1-a+ib}\right)\, .
\end{equation}
By the complex Frobenius Theorem, there exists a complex function $f$ and a holomorphic function $w$ locally on the manifold such that
\begin{equation}\label{eq:11}
\omega_2 - \frac{dz}{1-a+ib} = fdw\, .
\end{equation}
Substituting (\ref{eq:11}) into (\ref{eq:10}) yields 
\[df\wedge dw = f\  [-(a + ib)dz + (1 - a + ib)d\bar{z}]\wedge dw\, .\]
By defining $f = e^{-(a + ib)z + (1 - a + ib)\bar{z}} g$ for some function $g \neq 0$, we get $dg\wedge dw = 0$. Thus
\[\omega_2 - \frac{dz}{1-a+ib} = e^{-(a + ib)z + (1 - a + ib)\bar{z}} g(w)dw\, . \]
By defining $\tilde{w} = \int g(w)dw$ and dropping the tilde, we have
\[\omega_2 - \frac{dz}{1-a+ib} = e^{-(a + ib)z + (1 - a + ib)\bar{z}}dw\, . \]
Since $(\omega_1,\omega_2)$ is a coframing, $(z, w)$ can serve as a local holomorphic coordinate system on a open set. The coframing can be written as
\begin{equation*}
\begin{aligned}
\omega_1& = dz\, ,\\
\omega_2 &= \frac{dz}{1-a+ib} + e^{-(a + ib)z + (1 - a + ib)\bar{z}}dw\, .
\end{aligned}
\end{equation*}

We will analyze the symmetry group of the coframing. Define
\begin{equation*}
G_{a,b} =\{(\alpha,\ \beta)| -(a + ib)\alpha + (1 - a + ib)\bar{\alpha}= 2k\pi i, \ \ \ \text{where }\alpha,\ \beta\in\mathbb{C} \text{ and } k\in\mathbb{Z}\}\, ,
\end{equation*}
where $G_{a,b}$ acts on the local coordinate as $(z,\ w)\rightarrow (z + \alpha,\ w + \beta\,) $. We will analyze the elements of $G_{a,b}$. Let $\alpha = \alpha _ 0 + i \alpha_1$, where $\alpha_0, \alpha_1\in \mathbb{R}$. We have
\begin{equation}\label{eq:12}
\begin{aligned}
&(1- 2a)\alpha_0 + 2b\alpha_1 = 0\, ,\\
 &\alpha_1=2k\pi\, ,
\end{aligned}
\end{equation}
where $k\in \mathbb{Z}$.

\noindent For different $a,\ b$, there exist three families of solution for $\alpha$:
\begin{enumerate}
\item if $a\neq \frac{1}{2}$, then $\alpha_0 = \frac{-4bk\pi}{1 - 2a}, \ \alpha_1 = 2k\pi$. Define
\[\Gamma_1 =\left \{ \left.\left(\left( \frac{-4b\pi}{1 - 2a} + i2\pi\right)k, \beta_0 + i\beta_1\right)\right| k,\ \beta_0,\ \beta_1\in \mathbb{Z}\right \}\, .
\]
We can get a non-compact quotient
\[\mathbb{C}^2/\Gamma_1\cong \mathbb{R}\times S^1\times T^2
\]
that supports a homogeneous complex Engel structure.
\item if $a= \frac{1}{2}, \ b = 0$, then $\alpha_1 = 2k\pi$. Define
\[\Gamma_2 =\left \{ \left.\left(\alpha_0 + i2k\pi, \beta_0 + i\beta_1 \right)\right| k,\ \alpha_0,\ \beta_0,\ \beta_1\in \mathbb{Z}\right \}\, .\]
We get a compact quotient $\mathbb{C}^2/\Gamma_2$ that supports a homogeneous complex Engel structure.

In this case, we can define $\theta_1 = - \omega_1$ and $\theta_2 = \omega_1 - \tfrac{1}{2}\omega_2$. The structure equation is
\begin{equation}
\begin{aligned}
d\theta_1 &= 0\, ,\\
d\theta_2 &= \tfrac{1}{2}\, (\theta_1 - \bar\theta_1)\wedge\theta_2\, .
\end{aligned}
\end{equation}
Define $\theta_1 = \alpha + i \beta$ and $\theta_2 = \gamma + i\delta$ for real parts and imaginary parts decomposition. We have
\begin{equation}
\begin{aligned}
d\alpha &= 0\, ,\\
d\beta &= 0\, ,\\
d\gamma &= -\beta\wedge\delta\,,\\
d\delta &= \beta\wedge\gamma\, .
\end{aligned}
\end{equation}
Thus the Lie algebra is a 4-dimensional solvable Lie algebra $\mathfrak{g}$ with nontrivial brackets:
 \[[X, Y] = Z,\, [X, Z] = -Y,\]
where $X, Y, Z \in \mathfrak{g}$. By the classification results in \cite{MR3480018}, the corresponding connected and simply-connected Lie group has a co-compact lattice.
\item if $a= \frac{1}{2}, \ b\neq 0$, then $\alpha_1 = 0$.  Define
\[\Gamma_3 =\left \{\left.\left (\alpha_0, \beta_0 + i\beta_1 \right)\right| \alpha_0, \beta_0,\ \beta_1\in \mathbb{Z}\right \}\, .\]
Then we can get a non-compact quotient
\[\mathbb{C}^2/\Gamma_3\cong \mathbb{R}^1\times S^1\times T^2\, .\]
\end{enumerate}

In summary, there exists a compact quotient of type $C1$ that can support a homogeneous complex Engel structure if and only if $a = \frac{1}{2}$ and  $b = 0$.

\subsection{Homogeneous Case C2}

Now, $(p_1, p_2, q_1, q_2, r_1, r_2)=(\tfrac{1}{2} + ib,\ 0,\ 2ia,\  0,\ 0,\ 0)$. Assume $a\neq 0$, otherwise, it is a special case of $C1$. The structure equation is
\begin{equation}\label{eq:16}
\begin{aligned}
d\omega_1& = 2ia\bar{\omega}_1\wedge \omega_1\, ,\\
d\omega_2& = \left(\omega_1 - \omega_2\right)\wedge\bar{\omega}_1  - \left[\left(\tfrac{1}{2}+ ib\right)\omega_1 +  \left(\tfrac{1}{2} - ib\right)\bar{\omega}_1\right]\wedge\omega_2\, .
\end{aligned}
\end{equation}

By (\ref{eq:16})
\begin{equation*}
d\left(\omega_1 +  \left(-\tfrac{1}{2} + i\left(2a-b\right)\right)\omega_2\right) = \left(\tfrac{1}{2} + ib\right)\left(\bar{\omega}_1 - \omega_1\right)\wedge \left(\omega_1 + \left(-\tfrac{1}{2} + i\left(2a-b\right)\right)\omega_2\right)\, .
\end{equation*}
By the complex Frobenius Theorem, there exist complex functions $p$ and $q$, and holomorphic functions $z$ and $w$ such that
\begin{equation*}
\begin{aligned}
\omega_1 &= pdz\, ,\\
\omega_1 + \Big(-\frac{1}{2} + i(2a-b) \Big)\omega_2 &= qdw\, .
\end{aligned}
\end{equation*}
To calculate the function $p$, write $\omega_1 = \alpha + i\beta$, where $\alpha$ and $\beta$ are the real and imaginary part of $\omega_1$, respectively. From the structure equation
\begin{equation*}
\begin{aligned}
&d\alpha = -4a\alpha\wedge \beta\, ,\\
&d\beta = 0\, .
\end{aligned}
\end{equation*}
Thus there exist functions $f,\ x,\ y$ such that $\alpha = fdy, \beta = dx$ and $df\equiv 4af dx\ \mod dy$. After redefining $y$, we can write $\alpha = e^{4ax}dy$. Thus
\begin{equation}
\omega_1 = e^{4ax}d\left(y - \tfrac{i}{4a} e^{-4ax}\right)\, .
\end{equation}

Define $z = y - \tfrac{i}{4a} e^{-4ax}$ as a local holomorphic coordinate ( Note: $\Im(z) \neq 0$. So, according to the sign of $a$, we can restrict the definition of $z$ to half of the complex plane), then
\[\omega_1 = \frac{-i}{2a(z - \bar{z})}dz\, .\]
From the structure equation, we get
\[dq\wedge dw = \Big(\frac{1}{2} + ib\Big)\frac{i}{2a(\bar{z} - z)}d(\bar{z} - z) \wedge q dw\]
Thus
\[dq\equiv \Big(\frac{1}{2} + ib\Big)\frac{i}{2a(\bar{z} - z)}d(\bar{z} - z) q  \ \ \ \ \ \mod\ \ \ \ dw\]
After redefining $w$, we can take $q = (\bar{z} - z)^{\tfrac{-2b +i}{4a}}$. So in the local holomorphic coordinate system $z,\ w$,
 \begin{equation*}
\begin{aligned}
\omega_1 & = \frac{-i}{2a(z - \bar{z})}dz\, ,\\
\omega_2 & =\frac{2}{-1 + 2(2a-b)i}\left((\bar{z} - z)^{\frac{-2b + i}{4a}} dw + \frac{i}{2a(z - \bar{z})}dz \right)\, .
\end{aligned}
\end{equation*}

%
%
%

\subsubsection{Compact case of case $C2$}
By (\ref{eq:16}),
\begin{equation}\label{eq:103}
d(\omega_1\wedge\bar{\omega}_2\wedge\omega_2) = 2i(a +b)\omega_1\wedge\bar{\omega}_1\wedge\omega_2\wedge\bar{\omega}_2\, .
\end{equation}
If $a +b \neq 0$, the volume form is exact. By Stokes' Theorem, there cannot be a compact quotient that supports a homogeneous complex Engel structure when $a +b \neq 0$. In the following, only consider $a + b = 0$.  Define $\theta_1 = \omega_1$.   The structure equation is
\begin{equation}\label{eq:78}
\begin{aligned}
&d\theta_1 = 2ia\bar{\theta}_1\wedge \theta_1\, ,\\
&d\theta_2 = \left(\frac{1}{2} - ia\right)(\bar{\theta}_1 - \theta_1)\wedge\theta_2\, .
\end{aligned}
\end{equation}
Let $\theta_1 = \alpha + i\beta\, ,\theta_2 = \gamma + i\delta$ be real part and imaginary part decompositions. Then (\ref{eq:78}) is equivalent to
\begin{equation*}
\begin{aligned}
d\alpha &= -4a\alpha\wedge\beta\, ,\\
d\beta &= 0\, ,\\
d\gamma &= \beta\wedge\delta - 2a\beta\wedge\gamma\, ,\\
d\delta &= -\beta\wedge\gamma - 2a\beta\wedge\delta\, .
\end{aligned}
\end{equation*}
Let $X_1, X_2, X_3, X_4$ be left-invariant vector fields dual to the left-invariant forms $\alpha, \beta, \gamma, \delta$, respectively. Then the nontrivial brackets are
\begin{align*}
[X_2, X_1] &= 4aX_1\, ,\\
[X_2, X_4] &= X_3 - 2aX_4\, ,\\
[X_2, X_3] &= -2aX_3 - X_4\, .
\end{align*}
By \cite{MR3480018}, there exists a co-compact lattice for some $a$. We will calculate the conditions for the existence of a co-compact lattice.

Let $f = -2a + i$, $V = X_3 + iX_4$ and $W = X_3 - iX_4$. The nontrivial brackets are
\begin{align*}
[X_2, X_1] &= -(f+ \bar f)X_1\, ,\\
[X_2, V] &=  f\, V\, ,\\
[X_2, W] &= \bar f\, W\, .
\end{align*}
Since center of the Lie algebra $\mathfrak{g}$ is trivial, we have an exact sequence
\[0\rightarrow \mathfrak{g}\xrightarrow{ad} End(\mathfrak{g}),\]
where $\mathfrak{g}\xrightarrow{ad} End(\mathfrak{g})$ is the adjoint representation.
Let $X = xX_1 + yX_2 + z V + \bar{z} W$ be an element of $\mathfrak{g}$. Then
\begin{equation*}
ad(X) (X_1, V, W, X_2) = (X_1, V, W, X_2)\left[ {\begin{array}{cccc}
                                    -(f + \bar f)y & 0 & 0 & (f + \bar f)x\\
                                   0 & f y & 0 & -f z\\
                                   0 & 0 & \bar{f}y & -\bar{f}\bar{z}\\
                                   0 & 0 & 0 & 0\\
                                    \end{array} } \right]\, .
\end{equation*}
The connected and simply-connected Lie group corresponding to the Lie algebra $\mathfrak{g}$ is
\begin{equation*}
G\cong\left\{\left. \left[ {\begin{array}{cccc}
                                   y^{ -(f + \bar f)} & 0 & 0 & x\\
                                   0 & y^f & 0 & z\\
                                   0 & 0 & y^{\bar{f}} &\bar{z}\\
                                   0 & 0 & 0 & 1
                                    \end{array} } \right]
                                    \right | x,y \in\mathbb{R}, y > 0, z\in\mathbb{C}\right\}.
\end{equation*}

\begin{proposition}
If there exists a complex number $\lambda$ and a matrix $A\in SL_3(\mathbb{Z})$ such that
\begin{enumerate}
\item  $\left|\lambda\right|\neq 1$
\item the eigenvalues of $A$ are $(\lambda\bar\lambda)^{-1}, \lambda, \bar{\lambda}$
\item there exists $k\in\mathbb{Z}$ such that $-\tfrac{1}{2a} \log( \left|\lambda\right|) = \arg{\lambda} + 2k\pi$
\end{enumerate}
then there exists a co-compact lattice $\Gamma$ such that $G/\Gamma$ supports a homogeneous complex Engel structure.
\end{proposition}

\begin{proof}
Let $N\subset G$ be the subgroup
\begin{equation}
N = \left\{
\left.\left[ {\begin{array}{cccc}
                              
                                  1 & 0 & 0 &  r \\
                                  0 & 1  & 0 & z\\
                                  0 & 0 & 1 & \bar{z}\\
                                  0 & 0 & 0 & 1
                                    \end{array} } \right] \right\vert r\in \mathbb{R}, z\in\mathbb{C}\right\}\, .
\end{equation}
It is easy to verify that $N$ is a normal subgroup of $G$.
Thus
\[G/N \cong  \left\{
\left.\left[ {\begin{array}{cccc}
                               y^{ -(f + \bar f)} & 0 & 0 & 0\\
                                   0 & y^f & 0 & 0\\
                                   0 & 0 & y^{\bar{f}} & 0\\
                                   0 & 0 & 0 & 1
                                    \end{array} } \right]
                                    \right | y \in\mathbb{R}, y > 0\right\}  
\]
 is a quotient group.                                   
Let $L_1 = \langle \vec{v}_1, \vec{v}_2 ,\vec{v}_3 \rangle$ be a lattice of the normal subgroup $N$, to be determined later. We need to find a lattice $L_2$ of $G/N$ such that the lattice of the group $G$ is
\[L  = \left\{
\left.\left[ {\begin{array}{cc}
                              
                                  \gamma & \vec{v} \\
                                  0 & 1\\
                                    \end{array} } \right] \right\vert \text{ where } \gamma\in L_2, \vec{v}\in L_1\right\}\, .
\]

By the multiplication rule of the group $G$, this is equivalent to $\gamma \vec{v}\in L_1$ for any $\gamma \in L_2$ and any $\vec{v} =  \left[ {\begin{array}{c}v_1 \\ v_2 \\ v_3\end{array} } \right] \in L_1$. Hence we need to find $a_{ij}\in \mathbb{Z}$  and $c > 0$ and $c\neq 1$ such that
\begin{align}\label{eq:79}
\gamma_c \vec{v}_1 &= a_{11} \vec{v}_1 + a_{21} \vec{v}_2 + a_{31} \vec{v}_3\, ,\nonumber\\                             
\gamma_c \vec{v}_2 &= a_{12} \vec{v}_1 + a_{22} \vec{v}_2 + a_{32} \vec{v}_3\, ,\\
\gamma_c \vec{v}_3 &= a_{13} \vec{v}_1 + a_{23} \vec{v}_2 + a_{33} \vec{v}_3\, ,\nonumber                            
\end{align}
where $\gamma_c$ is the linear transform with transformation matrix $\left[ {\begin{array}{ccc}
                                   c^{ -(f + \bar f)} & 0 & 0\\
                                   0 & c^f & 0\\
                                   0 & 0 & c^{\bar{f}}\\
                                    \end{array} } \right]$. Since $\langle\gamma_c \vec{v}_1, \gamma_c \vec{v}_2, \gamma_c \vec{v}_3\rangle$ will be a new basis for the lattice $L_1$, then  
\[A = \left[ {\begin{array}{ccc}    
                                  a_{11} & a_{12} & a_{13} \\
                                  a_{21} & a_{22} & a_{23}\\
                                  a_{31} & a_{32} & a_{33}\\
                                    \end{array} } \right] \in SL_3(\mathbb{Z}).
\]
                                    
(\ref{eq:79}) is equivalent to
\begin{equation}
\left[ {\begin{array}{ccc}
                                   c^{ -(f + \bar f)} & 0 & 0\\
                                   0 & c^f & 0\\
                                   0 & 0 & c^{\bar{f}}\\
                                    \end{array} } \right] = (\vec{v}_1, \vec{v}_2, \vec{v}_3) \left[ {\begin{array}{ccc}    
                                  a_{11} & a_{12} & a_{13} \\
                                  a_{21} & a_{22} & a_{23}\\
                                  a_{31} & a_{32} & a_{33}\\
                                    \end{array} } \right]  (\vec{v}_1, \vec{v}_2, \vec{v}_3)^{-1}
\end{equation}                                    
The eigenvalues of the matrix $A$ should be $c^{ -(f + \bar f)}, c^f$ and $c^{\bar{f}}$ for some $c > 0$.  Assume the eigenvalues are $(\lambda\bar\lambda)^{-1}, \lambda, \bar{\lambda}$ and fix
\begin{equation}\label{eq:80}
c^{-2a + i} = \lambda\, .\\
\end{equation}
Thus
\begin{equation}\label{eq:82}
c = \left|\lambda\right|^{-\tfrac{1}{2a}}\, .
\end{equation}
By (\ref{eq:80}),
\begin{equation}
\left|\lambda\right| \times \left|\lambda\right| ^ {-\tfrac{1}{2a} i} = \left|\lambda\right|\times e^{i\arg{\lambda}}\, .
\end{equation}
Thus the eigenvalue $\lambda$ and the parameter $a$ satisfy
\begin{equation}\label{eq:81}
-\tfrac{1}{2a} \log( \left|\lambda\right|) = \arg{\lambda} + 2k\pi 
\end{equation}
for certain $k\in\mathbb{Z}$.
\end{proof}

We will calculate an explicit condition on the existence of co-compact lattice. Since the eigenvalues are $(\lambda\bar\lambda)^{-1}, \lambda, \bar{\lambda}$, the characteristic polynomial of the matrix $A$ is
\begin{equation}
\begin{aligned}
&(x - (\lambda\bar\lambda)^{-1})(x - \lambda)(x - \bar{\lambda})\\
&= x^3 - \left(\left(\lambda\bar\lambda\right)^{-1} + \lambda + \bar{\lambda}\right)x^2 + \left(\lambda\left(\lambda\bar\lambda\right)^{-1} + \bar{\lambda} \left(\lambda\bar\lambda\right)^{-1} + \lambda\bar{\lambda} \right)x - 1\, .
\end{aligned}
\end{equation}
Since $A \in SL_3(\mathbb{Z})$,  there exist $m, n \in \mathbb{Z}$ such that
\begin{equation}\label{eq:104}
\begin{aligned}
\left(\lambda\bar\lambda\right)^{-1} + \lambda + \bar{\lambda} &= m\, ,\\
\lambda\left(\lambda\bar\lambda\right)^{-1} + \bar{\lambda} \left(\lambda\bar\lambda\right)^{-1} + \lambda\bar{\lambda} &= n\, .
\end{aligned}
\end{equation}
Note if (\ref{eq:104}) satisfies, we can choose $A = \left[ {\begin{array}{ccc}
                                   0 & 0 & 1\\
                                   1 & 0& -n\\
                                   0 & 1 & m\\
                                    \end{array} } \right]\in SL_3(\mathbb{Z})$.                                    

Define $p = \lambda + \bar{\lambda}$ and $q = \left(\lambda\bar\lambda\right)^{-1}$. It is easy to verify that $p$ and $q$ are real numbers and $\tfrac{4}{q} \ge p^2$. To be a co-compact lattice, $0 < q <1$. Then by (\ref{eq:104}),
\begin{equation}\label{eq:105}
\begin{aligned}
p + q &= m\, ,\\
pq + \frac{1}{q} &= n\, .
\end{aligned}
\end{equation}
Then the eigenvalue is $\lambda = \frac{p}{2} + i\sqrt{\tfrac{1}{q} - \tfrac{p^2}{4}} $. 
\begin{remark}
There exist countably infinite families of solutions for $p$ and $q$, that yield infinitely many families of co-compact lattices. The co-compact lattices can be derived from solutions of (\ref{eq:105}).
\end{remark}

Now we give an example for some $a$ such that there exists a compact quotient.
\begin{example}
Let $\left[ {\begin{array}{ccc}
                                   0 & 0 & 1\\
                                   1 & 0& 1\\
                                   0 & 1 & 0\\
                                    \end{array} } \right]\in SL_3(\mathbb{Z})$. Define $s = (108 + 12\sqrt{69})^{\frac{1}{3}}$, then
 \[\lambda = -\frac{s}{12} - \frac{1}{s} + \frac{\sqrt{3}\left(\frac{s}{6} - \frac{2}{s} \right)}{2}i \, .\]
\noindent We can calculate $a$ by (\ref{eq:81}) and $c$ by (\ref{eq:82}).
\end{example}

\subsection{Homogeneous Case C3}
Now
\[(p_1, p_2, q_1, q_2, r_1, r_2) = \Big(\tfrac{1}{2} - ib,\  0,\  2ib,\ \tfrac{1}{2}(2b + i)(2ia - b),\ 2b^2 - ib,\ (b^2 + \tfrac{1}{4})(2a + ib)\Big)\]
Assume at least one of $a$ or $b$ is nonzero. Otherwise, it will be a special case of $C1$ with $a = \frac{1}{2}$ and $b = 0$. The structure equation is
\begin{equation}\label{eq:17}
\begin{aligned}
d\omega_1& = -[-2ib\bar{\omega}_1 + \frac{1}{2}(2b-i)(-2ia - b) \bar{\omega}_2]\wedge\omega_1\\ 
&\ \ \ -\left[\frac{1}{2}(2b + i)(2ia - b)\omega_1 + (2b^2 + ib)\bar{\omega}_1 + \left(b^2 + \frac{1}{4}\right)(2a - ib)\bar{\omega}_2\right]\wedge\omega_2\, ,\\
d\omega_2& = \omega_1\wedge\bar{\omega}_1 - \left(\frac{1}{2} - i b\right)(\omega_1 - \bar{\omega}_1)\wedge\omega_2\, .
\end{aligned}
\end{equation}
Define $\theta = \omega_1 + \left(-\frac{1}{2} + i b\right)\omega_2$. By (\ref{eq:17}), we have
\begin{equation}\label{eq:18}
d\theta =\left[\left(\tfrac{1}{2} + ib\right)\left(\bar{\omega}_1 + \left(2a - ib\right)\bar{\omega}_2\right) - \left(\tfrac{1}{2} + 2a\right)\omega_1\right]\wedge\theta\, .
\end{equation}

Since the symmetry groups are different for different parameters, we will consider the structure equation with different parameters: 
\begin{enumerate}
\item $a = -\tfrac{1}{4}$
\item $a = b^2 \neq 0$
\item $a \neq -\tfrac{1}{4}$ and $a \neq b^2$
\end{enumerate}
\subsubsection{ $a = -\frac{1}{4}$ }
The structure equation (\ref{eq:18}) reduces to 
\begin{align}\label{eq:19}
d\Big( \omega_1 + \Big(-\frac{1}{2} + i b\Big)\omega_2\Big) &= \Big(\frac{1}{2} + ib\Big)\overline{\Big(\omega_1 + \Big(-\frac{1}{2} + ib\Big)\omega_2\Big)}\\
&\ \ \ \ \ \ \wedge \Big(\omega_1 + \Big(-\frac{1}{2} + ib\Big)\omega_2\Big)\nonumber\, .
\end{align}

Let $\omega_1 + \Big(-\frac{1}{2} + ib\Big)\omega_2 = \alpha + i\beta$, where $\alpha$ and $\beta$ are real 1-forms. Then
\begin{equation}\label{eq:107}
\begin{aligned}
d\alpha &= -2b \alpha\wedge\beta\, ,\\
d\beta &=\alpha\wedge\beta\, .
\end{aligned}
\end{equation}
Since $d(\alpha + 2b\beta) = 0$, there exists function $x$ such that $\alpha + 2b\beta = dx$.
By the structure equation, we have $d\beta = dx\wedge \beta$, that implies the existence of a function $y$ such that $\beta = e^x dy$. Thus in terms of local coordinate $x,\ y$,
\begin{equation*}
\begin{aligned}
\alpha &= dx - 2be^x dy\, ,\\
\beta &=e^x dy\, .
\end{aligned}
\end{equation*}
So
\begin{equation*}
\omega_1 + \left(-\frac{1}{2} + ib\right)\omega_2 = e^x d(-e^{-x} - 2b y + i y)\, .
\end{equation*}
Let $z = -e^{-x} - 2b y + i y$. Then
\[\omega_1 + \Big(-\frac{1}{2} + ib\Big)\omega_2 = \frac{dz}{\Big(-\frac{1}{2} + i b\Big)z + \overline{\Big(-\frac{1}{2} + i b\Big)z}}\, .\]
By (\ref{eq:17}), we have
\begin{equation}\label{eq:87}
\begin{aligned}
d\omega_2& = \omega_1\wedge\bar{\omega}_1 - \left(\frac{1}{2} - i b\right)(\omega_1 - \bar{\omega}_1)\wedge\omega_2\\
&= \frac{\Big(-\frac{1}{2} + i b\Big) dz\wedge \omega_2}{\Big(-\frac{1}{2} + i b\Big) z + \overline{\Big(-\frac{1}{2} + i b\Big) z}} + \frac{-\overline{\Big(-\frac{1}{2} + i b\Big) } dz\wedge\bar{\omega}_2}{\Big(-\frac{1}{2} + i b\Big) z + \overline{\Big(-\frac{1}{2} + i b\Big) z}}\\
&\ \ \ \ \ + \frac{dz\wedge d\bar{z}}{\Bigg(\Big(-\frac{1}{2} + i b\Big) z + \overline{\Big(-\frac{1}{2} + i b\Big) z}\Bigg)^2}\, .
\end{aligned}
\end{equation}

\begin{lemma}\label{lemma}
Assume $\langle dz, \omega_2\rangle$ forms a Frobenius system. Then there exists a function $f$ and a holomorphic function $w$ such that
\[\omega_2 = dw + f dz\]
\end{lemma}

\begin{proof}
Since $\langle dz, \omega_2\rangle$ forms a Frobenius system of rank 2, there exist functions $u,\ v,\ a,\ b$ such that $dz = a du + b dv$ and $du, \ dv$ are linearly independent . Since $dz \neq 0$, at least one of $a, \ b$ is nonzero. Without loss of generality, assume $b\neq 0$. So $dv = \frac{1}{b} (dz - a du)$. So $v$ is a function of $z$ and $x$ and locally we take $z$ and $u$, instead of $v$ and $u$, as local coordinates.

By the complex Frobenius Theorem, there exist functions $r(z, u)$ and $s(z, u)$ such that
\begin{align*}
\omega_2 &= r(z, u) du + s(z, u) dz\\
         &=d\Big(\int r(z, u)\  du\Big) - \Big(\int \frac{\partial r(z, u)}{\partial z} du \Big) dz + s(z, u) dz\\
         &= d\Big(\int r(z, u)\  du\Big) + \Big[ s(z, u) - \Big(\int \frac{\partial r(z, u)}{\partial z} du \Big)  \Big] dz\, .
\end{align*} 
Since $\omega_2$ and $dz$ are linearly independent,  $d\Big(\int r(z, u)\  du\Big) \neq 0$. Let $ w = \Big(\int r(z, u)\  du\Big) $ and $f =  s(z, u) - \Big(\int \frac{\partial r(z, u)}{\partial z} du \Big) $, then
\[\omega_2 = dw + f dz\, .\]
\end{proof}

Let $D = -\frac{1}{2} + i b $. By (\ref{eq:87}) and Lemma \ref{lemma}, we get
\[df\wedge dz = \frac{dz}{Dz + \bar{D}\bar{z}} \wedge \left( D dw - \bar{D} d\bar{w} + \frac{d \bar{z}}{Dz + \bar{D}\bar{z}} - \bar{D}\bar{f} d\bar{z} \right)\, .\]
Thus
\begin{equation}\label{eq:66}
\begin{aligned}
\frac{\partial f}{\partial w} &= \frac{-D}{Dz + \bar{D}\bar{z}}\, ,\\
\frac{\partial f}{\partial \bar{w}} &= \frac{\bar{ D}}{Dz + \bar{D}\bar{z}}\, ,\\
\frac{\partial f}{\partial \bar{z}} &= \frac{\bar{D}\bar{f} - \frac{1}{Dz + \bar{D}\bar{z}}}{Dz + \bar{D}\bar{z}}\, .
\end{aligned}
\end{equation}

Let $f = f_1 + if_2$, where $f_1$ and $f_2$ are real and imaginary parts of $f$, respectively. Let $z = x+ i y$ and $w = u + i v$, then (\ref{eq:66}) is equivalent to
\begin{equation}
\begin{aligned}
\frac{\partial f_1}{\partial u} & = 0\, ,\\
\frac{\partial f_1}{\partial v} & = 0\, ,\\
\frac{\partial f_2}{\partial u} & = \frac{2b}{x+ 2by}\, ,\\
\frac{\partial f_2}{\partial v} & = -\frac{1}{x+ 2by}\, ,\\
\frac{\partial f_1}{\partial x} - \frac{\partial f_2}{\partial y}& = \frac{f_1 + 2bf_2 - \frac{2}{x + 2by}}{x + 2by}\, ,\\
\frac{\partial f_1}{\partial y} + \frac{\partial f_2}{\partial x} &= \frac{ 2bf_1 - f_2}{x + 2by}\, .
\end{aligned}
\end{equation}
So $f_1 = f_1(x, y), f_2 = g(x, y) + \frac{2b u}{x+ 2by} - \frac{v}{x+ 2by}$. The equation is equivalent to
\begin{equation}
\begin{aligned}
\frac{\partial f_1}{\partial x} - \frac{\partial g}{\partial y}& = \frac{f_1 + 2bg - \frac{2}{x + 2by}}{x + 2by}\, ,\\
\frac{\partial f_1}{\partial y} + \frac{\partial g}{\partial x} &= \frac{ 2bf_1 - g}{x + 2by}\, .
\end{aligned}
\end{equation}
Consider the differential ideal $I = \langle\theta_1,\ \theta_2\rangle$, where
\begin{equation}
\begin{aligned}
\theta_1 & =  df_1 - pdx - qdy\, ,\\
\theta_2 & = dg + \left(q -  \frac{ 2bf_1 - g}{x + 2by} \right) dx - \left(p -  \frac{f_1 + 2bg - \frac{2}{x + 2by}}{x + 2by} \right)dy\, .
\end{aligned}
\end{equation}
and
\begin{equation}
\begin{aligned}
d\theta_1 & = - \pi_1\wedge dx - \pi_2\wedge dy\, ,\\
d\theta_2 & = \pi_2\wedge dx - \pi_1 \wedge dy\, ,
\end{aligned}
\end{equation}
where $\pi_1 \equiv dp \ \mod\ (dx)$ and $\pi_2 \equiv dq \ \mod \ (dy)$. This system is involutive and its Cartan characters are $(s_1,s_2) = (2, 0)$. So the solution depends on 2 functions of 1 variable. 

We will calculate the coframing in local coordinate for  $b = 0$. Let $F_1 = F(-iz)$, $F_2 = G(i\bar{z})$ be two functions of one variable $z$ and $C$ be a constant. The general solution is of the following form
\begin{equation*}
\begin{aligned}
f_1 & = F_1' + F_2' + \left( \frac{4}{(z + \bar{z})^2} + C  \right)\frac{z + \bar{z}}{2}\, ,\\
f_2 & = iF_2' - iF_1' -\frac{2}{z + \bar{z}} \left(F_1 + F_2 + \frac{w - \bar{w}}{2i}\right)\, .
\end{aligned}
\end{equation*}

Thus in the case $b = 0$, the coframing is
\begin{equation*}
\begin{aligned}
\omega_1  &= \frac{1}{2}(dw + f dz) - \frac{2dz}{z + \bar{z}}\, ,\\
\omega_2 &= dw + f dz\, .
\end{aligned}
\end{equation*}

In our original parametrization, $z = -e^{-x} - 2b y + i y$. So $z + \bar{z} < 0$. Take a special form $F_1 = 0$ and $F_2 = 0$. Then
\[f(z, w) = \frac{2- (w - \bar{w})}{z + \bar{z}}\, .\]
The coframing can be written as
\begin{equation*}
\begin{aligned}
\omega_1  &= \frac{1}{2}\left(dw -  \frac{2+ (w - \bar{w})}{z + \bar{z}} dz\right)\, ,\\
\omega_2 &= dw +  \frac{2- (w - \bar{w})}{z + \bar{z}} dz\, .
\end{aligned}
\end{equation*}

We will prove that there exist compact quotients that support homogeneous complex Engel structures. Before proving this, we need to know the Lie algebra of the homogeneous complex Engel structures.
\begin{proposition}
There exists a basis $(X_1, X_2, X_3, X_4)$ such that the nontrivial brackets of the Lie algebra are
\begin{equation}
[X_2, X_3] = X_1,\, [X_2, X_4] = X_2, \, [X_3, X_4] = - X_3.
\end{equation}
\end{proposition}

\begin{proof}
Recall that
\begin{equation}
\begin{aligned}
d\alpha &= -2b \alpha\wedge\beta\, ,\\
d\beta &=\alpha\wedge\beta\, .
\end{aligned}
\end{equation}
Define $\omega_1 = \gamma + i\delta$. 
\begin{equation}\label{eq:108}
\begin{aligned}
d\gamma &= -2b\alpha\wedge\beta - \delta\wedge\beta + 2b\delta\wedge\alpha\, ,\\
d\delta &= \alpha\wedge\beta - \alpha\wedge\delta - 2b\beta\wedge\delta\, .
\end{aligned}
\end{equation}
Let $e_1, e_2, e_3, e_4$ be left-invariant vector fields dual to the left-invariant forms $\alpha, \beta, \gamma, \delta$, respectively. Then by (\ref{eq:108}), the nontrivial brackets are
\begin{align}
[e_1, e_2] &= -2be_1 + e_2 - 2be_3 + e_4\nonumber\\
[e_1, e_4] &= -2be_3 - e_4\\
[e_2, e_4] &= e_3 - 2be_4\nonumber\, .
\end{align}
We will consider 2 separate cases:
\begin{enumerate}
\item $b = 0$
\item $b \neq 0$
\end{enumerate}
If $b = 0$, the nontrivial brackets are
\begin{align}
[e_1, e_2] &= e_2 + e_4\nonumber\\
[e_1, e_4] &= - e_4\nonumber\\
[e_2, e_4] &= e_3\nonumber\, .
\end{align}
Define $X_1 = -2e_3, X_2 = e_4, X_3 = 2e_2 + e_4, X_4 = e_1$. Then
\begin{equation}
[X_2, X_3] = X_1,\, [X_2, X_4] = X_2, \, [X_3, X_4] = - X_3\, .
\end{equation}

If $b\neq 0$, define $\tilde e_1 = e_1 - \tfrac{1}{2b} e_2 - \tfrac{1}{2b}\left(e_4 - \tfrac{1}{2b}e_3\right), \tilde e_4 = e_4 - \tfrac{1}{2b}e_3$. The nontrivial brackets are
\begin{align}
[\tilde e_1, e_2] &= -2b\tilde e_1 + \tilde e_4 + \left(\tfrac{1}{2b} - 2b\right) e_3\nonumber\\
[\tilde e_1, \tilde e_4] &= - \left(2b + \tfrac{1}{2b}\right) e_3\nonumber\\
[e_2, \tilde e_4] &= -2b\tilde e_4\nonumber
\end{align}
Define $\tilde{\tilde e}_1 = \tilde e_1 - \tfrac{1}{4b}\tilde e_4$. Then the nontrivial brackets are
\begin{align}
[\tilde{\tilde e}_1, e_2] &= -2b\tilde{\tilde e}_1 + \left(\tfrac{1}{2b} - 2b\right) e_3\nonumber\\
[\tilde{\tilde e}_1, \tilde e_4] &= - \left(2b + \tfrac{1}{2b}\right) e_3\nonumber\\
[e_2, \tilde e_4] &= -2b\tilde e_4\nonumber
\end{align}
Define $X_1 = -\tfrac{4b^2 + 1}{2b} e_3, X_2 = \tilde{\tilde e}_1, X_3 = \tilde e_4, X_4 = \tfrac{1}{4b^2 + 1}\left(\left(-2b-\tfrac{1}{2b}\right) e_2 - \left(-2b+\tfrac{1}{2b}\right) e_4\right)$. Then
\begin{equation}
[X_2, X_3] = X_1,\, [X_2, X_4] = X_2, \, [X_3, X_4] = - X_3\, .
\end{equation}
Thus the theorem is true for both cases.
\end{proof}

%
By \cite{MR3480018}, there exists a co-compact lattice when $a = -\tfrac{1}{4}$. Thus there exists a compact quotient of type $C3$ when $a = -\tfrac{1}{4}$.

\subsubsection{$a = b^2 \neq 0$}
Define
\begin{equation}\label{eq:20}
\theta_1 =\omega_1 + (2b^2 + ib)\omega_2 + \frac{1}{4b^2 + 1}((4b^2 - 1) -4bi)\bar{\omega}_1 + \frac{b}{4b^2 + 1}(2b(4b^2 - 3) -(12b^2 - 1)i)\bar{\omega}_2
\end{equation}
Then the coframing $(\theta, \theta_1)$ satisfies
\begin{equation*}
\begin{aligned}
d\theta_1  &=0\, ,\\
d\theta &=  \left(-\tfrac{1}{2} + i b\right)\theta_1\wedge\theta\, .
\end{aligned}
\end{equation*}
Thus there exists a holomorphic coordinate system $(z, w)$ such that $\theta_1$ and $\theta$ can be written as linear combinations of $(dw,\ d\bar{w})$ and $dz$, respectively.  There exists a function $f(w,\bar{w})$  such that
\begin{equation*}
\begin{aligned}
\theta_1  &= df\, ,\\
\theta &=  e^{(-\frac{1}{2} + i b)f}dz\, .
\end{aligned}
\end{equation*}
Thus 
\[\omega_1 = e^{\left(-\tfrac{1}{2} + i b\right)f}dz - \left(-\tfrac{1}{2} + i b\right)\omega_2\, .\]
By (\ref{eq:20}), we have
\begin{align*}
&\ \ \ \ \Big(2b^2 - \frac{1}{2} - 2ib\Big)\bar{\omega}_2 + \Big(2b^2 + \frac{1}{2}\Big)\omega_2\\
&=df - e^{(-\frac{1}{2} + i b)f}dz- \tfrac{1}{4b^2 + 1}\Big((4b^2 - 1) -4bi\Big) e^{(-\frac{1}{2} - i b)\bar{f}}d\bar{z}\, .
\end{align*}
Let $\omega_2 = hdz + gdw$, where $h$ and $g$ are functions. By scaling, fix $g = 1$. Take $f = \Big(2b^2 + \frac{1}{2}\Big)w + \Big(2b^2 - \frac{1}{2} - 2ib\Big)\bar{w}$, then
\[h = -\frac{2e^{(-\frac{1}{2} + i b)\left[\Big(2b^2 + \frac{1}{2}\Big)w + \Big(2b^2 - \frac{1}{2} - 2ib\Big)\bar{w}\right]}}{4b^2 + 1}\, .\]
So
\[\omega_2 = -\frac{2e^{(-\frac{1}{2} + i b)\left[\left(2b^2 + \frac{1}{2}\right)w + \left(2b^2 - \frac{1}{2} - 2ib\right)\bar{w}\right]}}{4b^2 + 1}dz + dw\]
and
\[\omega_1 = e^{(-\frac{1}{2} + ib)\left[\left(2b^2 + \frac{1}{2}\right)w + \left(2b^2 - \frac{1}{2} - 2ib\right)\bar{w}\right]} dz - \left(-\frac{1}{2} + i b\right)\omega_2\, .\]
%
\begin{proposition}
There does not exist a compact quotient that supports a homogeneous complex Engel structure when $a = b ^2 \neq 0$.
\end{proposition}
\begin{proof}
After changing $\theta_1$ to $\frac{1}{\left(-\tfrac{1}{2} + i b\right)}\theta_1$, the structure equation is
\begin{equation}\label{eq:109}
\begin{aligned}
d\theta_1  &=0\, ,\\
d\theta &= \theta_1\wedge\theta\, .
\end{aligned}
\end{equation}
Define $\theta_1 = \alpha + i\beta, \theta = \gamma + i\delta$ as the real and imaginary parts decompositions. Then (\ref{eq:109}) is equivalent to
\begin{equation}
\begin{aligned}
d\alpha &= 0\, ,\\
d\beta &= 0\, ,\\
d\gamma &= \alpha\wedge\gamma - \beta\wedge\delta\, ,\\
d\delta &= \alpha\wedge\delta + \beta\wedge\gamma\, .
\end{aligned}
\end{equation}
Let $-X_3, X_4, X_1, X_2$ be left-invariant vector fields dual to the left-invariant forms $\alpha, \beta, \gamma, \delta$, respectively. Then the nontrivial brackets are
\begin{align*}
[X_1, X_3] &= X_1\, ,\\
[X_2, X_3] &= X_2\, ,\\
[X_1, X_4] &= -X_2\, ,\\
[X_2, X_4] &= X_1\, .
\end{align*}
The Lie algebra is a solvable Lie algebra, denoted by $\mathfrak{g}_{4,10}$ in \cite{MR0153714}. According to the classification results of the existence of co-compact lattices for 4-dimensional solvable Lie groups in \cite{MR3480018}, the connected and simply-connected Lie group corresponding to $\mathfrak{g}_{4,10}$ does not have a co-compact lattice. Therefore, there does not exist a compact quotient that supports a homogeneous complex Engel structure in this case.
\end{proof}

\subsubsection{ $a \neq -\frac{1}{4}\ \ \text{and}\ \ a\neq b^2$ }
Let $\omega_1 = \alpha + i\beta$ and $\omega_2 = \gamma + i \delta$. From the structure equation (\ref{eq:17}), we have
\begin{equation}\label{eq:73}
\begin{aligned}
d\alpha &= -4b\alpha\wedge\beta -2b^2\alpha\wedge\gamma + b\alpha\wedge\delta + (4ab - 2b) \beta\wedge\gamma - (2a + 4b^2) \beta
\wedge\delta\\
& - 2b\left(\frac{1}{4} + b^2\right)\gamma\wedge\delta\, ,\\
d\beta &= -4ab\alpha\wedge\gamma + 2a\alpha\wedge\delta + 2b^2 \beta\wedge\gamma - b\beta\wedge\delta - 4a\left(\frac{1}{4} + b^2\right)\gamma\wedge\delta\, ,\\
d\gamma &= \beta\wedge(\delta - 2b\gamma)\, ,\\
d\delta &= \beta\wedge(2\alpha - \gamma - 2b\delta)\, .
\end{aligned}
\end{equation}

Define $\tilde{\alpha} = -2a \alpha + b\beta -(4a^2 - b^2)\gamma + 4ab\delta$. Since at least one of $a$ or $b$ is not zero, $\tilde{\alpha} \neq 0$ and $d\tilde{\alpha} = 0$.
Let $e_1,\ e_2,\ e_3,\ e_4$ be the dual vector fields of the 1-forms $\alpha,\beta,\gamma,\delta$, respectively. 

\begin{remark}
From Lie theory, there is a split exact sequence \cite{MR1153249}
\[ 0\rightarrow Rad(\mathfrak{g})\rightarrow \mathfrak{g}\rightarrow \mathfrak{g}/Rad(\mathfrak{g})\rightarrow 0\, ,\]
where $Rad(\mathfrak{g})$ is the radical ideal of $\mathfrak{g}$. Thus
\[ \mathfrak{g}\cong Rad(\mathfrak{g})\oplus \mathfrak{g}/Rad(\mathfrak{g})\, .\]
\end{remark}
Denote $\mathfrak{g}_1 = \mathfrak{g}/Rad(\mathfrak{g})$. We will prove the following theorem:

\begin{theorem}
For the Lie algebra $\mathfrak{g}$ corresponding to the structure equation(\ref{eq:73}), $Rad(\mathfrak{g})$ is 1-dimensional and $\mathfrak{g}_1$ is simple 3-dimensional Lie algebra. Specially, 
 \begin{itemize}
 \item $(a<-\frac{1}{4})$ or $(0\leq a < b^2)$ or $(b = 0\ \ \text{and}\ \ a<0)$, the Lie algebra is $\mathbb{R}\times\mathfrak{sl}(2,\mathbb{R})$.
 \item  $(-\frac{1}{4} < a < 0)$ or $(a > b^2)$ or $(b = 0\ \ \text{and}\ \ \ a >0)$, the Lie algebra is $\mathbb{R}\times\mathfrak{su}(2)$.
  \end{itemize}
\end{theorem}

\begin{proof}
We will prove this theorem by analyzing the result for the following 3 cases:
\begin{itemize}
\item 1. $b = 0$

In this case $\tilde{\alpha} = \alpha + 2a\gamma$. Define $\tilde{\beta} = \beta,\ \tilde{\gamma} = -2\alpha + \gamma,\ \tilde{\delta} = \delta$, then
\begin{equation}\label{eq:110}
\begin{aligned}
d\tilde{\beta} &= -a\tilde{\gamma}\wedge\tilde{\delta}\, ,\\
d\tilde{\gamma} &= -(1+ 4a) \tilde{\delta}\wedge\tilde{\beta}\, ,\\
d\tilde{\delta} &= -\tilde{\beta}\wedge \tilde{\gamma}\, .
\end{aligned}
\end{equation}

If $a > 0$, we define $\tilde{\tilde{\beta}} = \tfrac{1}{\sqrt{1 + 4a}}\tilde{\beta}, \tilde{\tilde{\gamma}} = \tfrac{1}{\sqrt{a}}\tilde{\gamma},\tilde{\tilde{\delta}} = \tfrac{1}{\sqrt{a(1 + 4a)}}\tilde{\delta}$. (\ref{eq:110}) is equivalent to
\begin{equation*}
\begin{aligned}
d\tilde{\tilde{\beta}} &= \tilde{\tilde{\gamma}}\wedge\tilde{\tilde{\delta}}\, ,\\
d\tilde{\tilde{\gamma}} &=\tilde{ \tilde{\delta}}\wedge\tilde{\tilde{\beta}}\, ,\\
d\tilde{\tilde{\delta}} &= \tilde{\tilde{\beta}}\wedge \tilde{\tilde{\gamma}}\, .
\end{aligned}
\end{equation*}
Thus $\tilde{\tilde{\beta}},\ \tilde{\tilde{\gamma}},\ \tilde{\tilde{\delta}}$ are left-invariant forms of the Lie group $SU(2)$. So if $a > 0$, the manifold can be taken as $S^1\times SU(2)$.

If $a < 0$, $Rad(\mathfrak{g}) = \{e_1 + 2e_3\}$ and $\mathfrak{g}/Rad(\mathfrak{g}) = \{u =e_1 - \frac{1}{2a}e_3, v=\ e_2, w =\ e_4\}$. Define
\begin{equation*}
\begin{aligned}
H&= \frac{-4\sqrt{-a}}{1 + 4a}u\, ,\\
X&=\sqrt{-a}v + w\, ,\\
Y&= \frac{\sqrt{-a}}{a(1 + 4a)}v -\frac{1}{a(1 + 4a)}w\, .
\end{aligned}
\end{equation*}
The nontrivial brackets are
\begin{align*}
[H, X]&= 2X\, ,\\
[H,Y]&=-2Y\, ,\\
[X,Y]&=H\, .
\end{align*}
$(H, X, Y)$ forms a canonical basis for the Lie algebra $\mathfrak{sl}(2, \mathbb{R})$. Since  $SL_2(\mathbb{R})$ has  co-compact lattices \cite{MR2655311},  in this case, there exists a compact quotient that supports homogeneous complex Engel structure.

\item 2. $a =0$

By our assumption, $b\neq 0$. The radical ideal is $Rad(\mathfrak{g}) = \left\{e_1 + \frac{2e_3 + 4be_4}{1 + 4b^2}\right\}$ and $\mathfrak{g}/Rad(\mathfrak{g})= \{u =e_1, v=\ e_2 - \frac{e_3}{b}, w =\ e_4\}$. The nontrivial brackets are
\begin{align*}
[u,v]&= 2bu + 2w\\
[u, w] &=-bu\\
[v, w] &= \left(2b^2 - \frac{1}{2}\right)u + bv + 2bw
\end{align*}
Define
\begin{equation*}
\begin{cases}
A= \frac{2b-1}{4b^2},\ B = \frac{1}{2b},\ C = \frac{1}{2b^2} , D= 2b+1,\ E=-2b, F=2, s = 2, \ \ \text{ if } b \neq -\frac{1}{2}\, ,\\
A= \frac{2b+1}{4b^2},\ B = -\frac{1}{2b},\ C = \frac{1}{2b^2}, D= 2b-1,\ E=2b, F=2, s = -2,\ \ \text{ if } b \neq \frac{1}{2}\, .
\end{cases}
\end{equation*}
and
\begin{align*}
H&=sv\, ,\\
X&=Au + Bv + Cw\, ,\\
Y&= Du + Ev + Fw\, .
\end{align*}
Then $(H, X, Y)$ forms a canonical basis for the Lie algebra $\mathfrak{sl}(2, \mathbb{R})$ such that
\begin{align*}
[H, X]&= 2X\, ,\\
[H,Y]&=-2Y\, ,\\
[X,Y]&=H\ .
\end{align*}
Since  $SL_2(\mathbb{R})$ has co-compact lattices \cite{MR2655311}, in this case there exists a compact quotient that supports a homogeneous complex Engel structure.

\item 3. $a \neq 0, a \neq -\tfrac{1}{4} \ \ \text{and}\ \ b\neq 0$

The radical ideal is $Rad(\mathfrak{g}) = \left\{e_1 + \frac{2e_3 + 4be_4}{1 + 4b^2}\right\}$ and 
\[\mathfrak{g}/Rad(\mathfrak{g}) = \left\{u =e_1 +\frac{e_4}{2b}, v=\ e_2 - \frac{e_4}{4a}, w =\ e_3+ \frac{4a^2 - b^2}{4ab}e_4\right\}.\]
The nontrivial brackets are
\begin{align*}
[u,v]&= -\frac{-8ab^2 + 4a^2 - b^2}{4ab}u + \frac{1}{2b}w\, ,\\
[u, w] &=-(-4ab^2 +4a^2 - b^2 +a)\left(\frac{1}{4a}u + \frac{1}{2b}v\right)\, ,\\
[v, w] &= \frac{-4b^4 + 16a^3 +12ab^2 + b^2}{8ab}u + \frac{-4ab^2 +4a^2 - b^2 +a}{4a}v\\ 
&\quad -\frac{-8ab^2 + 4a^2 -b^2}{4ab}w\, .
\end{align*}
After the proof of the following proposition, we will finish the proof of the theorem. 
\begin{proposition}
If $a \neq 0, a \neq -\tfrac{1}{4} \ \text{and}\ \ b\neq 0$, there exists a compact quotient that supports a homogeneous complex Engel structure.
\end{proposition}

\begin{proof}
\textbf{1. $a>b^2$}

Define
\begin{align*}
U&=\sqrt{\frac{1}{-4ab^2 + 4a^2-b^2+a}}\, ,\\
V&= \sqrt{\frac{a}{2(-4ab^2 + 4a^2-b^2+a)}}\\
\end{align*}
and 
\begin{align*}
A&=\frac{V}{2a}(-8ab^2U + 4a^2U -b^2U+b),\ B = V,\ C = -UV,\\
D&= -\frac{V}{2a}(-8ab^2U + 4a^2U -b^2U-b),\ E=V, F=UV,\\
 s&=2bU\, .
\end{align*}
\noindent Then define
\begin{align}
H&=sv\, ,\nonumber\\
X&=Au + Bv + Cw\, ,\\
Y&= Du + Ev + Fw\, .\nonumber
\end{align}
Then $(H, X, Y)$ forms a canonical basis for the Lie algebra $\mathfrak{su}(2)$ with the following nontrivial brackets
\begin{align*}
[H, X]&= Y\, ,\\
[H,Y]&=-X\, ,\\
[X,Y]&=H\, .
\end{align*}
Since $SU(2)$ has co-compact lattices, there exists a compact quotient that supports a homogeneous complex Engel structure.

\textbf{2. $0<a<b^2$}

Define
\begin{align*}
U&=\sqrt{\frac{-1}{-4ab^2 + 4a^2-b^2+a}}\, ,\\
V&= \sqrt{\frac{-a}{-4ab^2 + 4a^2-b^2+a}}\\
\end{align*}
and
\begin{align*}
A&=-\frac{V}{2a}(-8ab^2U + 4a^2U -b^2U-b),\ B = V,\ C =UV,\\
D&= \frac{V}{2a}(-8ab^2U + 4a^2U -b^2U+b),\ E=V, F=-UV,\\
 s&=4bU.
\end{align*}

\noindent Then define
\begin{align*}
H&=sv\, ,\\
X&=Au + Bv + Cw\, ,\\
Y&= Du + Ev + Fw\, .
\end{align*}
Then $(H, X, Y)$ forms a canonical basis for the Lie algebra $\mathfrak{sl}(2,\mathbb{R})$ with nontrivial brackets
\begin{align*}
[H, X]&= 2X\, ,\\
[H,Y]&=-2Y\, ,\\
[X,Y]&=H\, .
\end{align*}
Since $SL_2(\mathbb{R})$ has co-compact lattices \cite{MR2655311}, there exists a compact quotient that supports a homogeneous complex Engel structure.

\textbf{3. $-\frac{1}{4}<a<0$}

Define
\begin{align*}
U&=\sqrt{\frac{1}{-4ab^2 + 4a^2-b^2+a}}\, ,\\
V&= \sqrt{\frac{a}{2(-4ab^2 + 4a^2-b^2+a)}}\\
\end{align*}
and
\begin{align*}
A&=\frac{V}{2a}(-8ab^2U + 4a^2U -b^2U+b),\ B = V,\ C = -UV,\\
D&= -\frac{V}{2a}(-8ab^2U + 4a^2U -b^2U-b),\ E=V, F=UV,\\
 s&=2bU\, .
\end{align*}
\noindent Then define
\begin{align*}
H&=sv\, ,\\
X&=Au + Bv + Cw\, ,\\
Y&= Du + Ev + Fw\, .
\end{align*}
Note $Y = \overline{X}$. $(H, X, Y)$ forms a canonical basis of the Lie algebra $\mathfrak{su}(2)$ with nontrivial brackets
\begin{align*}
[H, X]&=  \overline{X}\, ,\\
[H, \overline{X}]&=-X\, ,\\
[X, \overline{X}]&=H\, .
\end{align*}
Since  $SU(2)$ has co-compact lattices, there exists a compact quotient that supports a homogeneous complex Engel structure.

\textbf{4. $a<-\frac{1}{4}$}

Define
\begin{align*}
U&=\sqrt{\frac{1}{-4ab^2 + 4a^2-b^2+a}}\, ,\\
V&= \sqrt{\frac{-a}{2(-4ab^2 + 4a^2-b^2+a)}}\\
\end{align*}
and
\begin{align*}
A&=\frac{V}{2a}(-8ab^2U + 4a^2U -b^2U+b),\ B = V,\ C =-UV,\\
D&= -\frac{V}{2a}(-8ab^2U + 4a^2U -b^2U-b),\ E=V, F=UV, \\
 s&=2bU.
\end{align*}
\noindent Then define
\begin{align*}
H&=sv\, ,\\
X&=Au + Bv + Cw\, ,\\
Y&= Du + Ev + Fw\, .
\end{align*}
$(H, X, Y)$ forms a canonical basis of the Lie algebra $\mathfrak{sl}(2,\mathbb{R})$ with nontrivial brackets
\begin{align*}
[H, X]&= Y\, ,\\
[H,Y]&=-X\, ,\\
[X,Y]&=-H\, .
\end{align*}
Since $SL_2(\mathbb{R})$ has co-compact lattices \cite{MR2655311}), there exists a compact quotient that supports a homogeneous complex Engel structure.
\end{proof}
\end{itemize}

Since there exists a co-compact lattice for each case, we have proved the theorem.
\end{proof}

\subsection{Homogeneous Case C4}
Now $(p_1, p_2, q_1, q_2, r_1, r_2) = (0,\ a-ib,\ 2ib,\ a-ib,\ -a-ib,\ 0)$. The structure equation is
\begin{align}\label{eq:21}
d\omega_1& = 2ib\bar{\omega}_1\wedge \omega_1 -(a+ib)\bar{\omega}_2\wedge\omega_1 +(a-ib)\bar{\omega}_1\wedge\omega_2,\nonumber\\
d\omega_2& = (\omega_1 - \omega_2)\wedge\bar{\omega}_1  - (a + ib)\bar{\omega}_2\wedge\omega_2.
\end{align}

\noindent It is easy to verify that
\[d(\omega_1\wedge\bar{\omega}_2\wedge\omega_2) = (1 + 2bi)\omega_1\wedge\bar{\omega}_1\wedge\omega_2\wedge\bar{\omega}_2\neq 0.\]
By Stokes' Theorem, there is no compact quotient of type $C4$ that supports a homogeneous complex Engel structure.

We will find local coordinate representations of $\omega_1$ and $\omega_2$ under different conditions for $a$ and $b$. The symmetry groups of the coframing are different for different $a$ and $b$. We define $A = 1$ and
\begin{equation}
B = 
\begin{cases}
-\frac{1}{2} \pm c +ib, \text{ where } c = \sqrt{\frac{1}{4}- (b^2 + a)}, \text{ if }b^2 + a \le \frac{1}{4}\\
-\frac{1}{2} \pm ic +ib, \text{ where } c = \sqrt{ (b^2 + a) - \frac{1}{4}},  \text{ if }b^2 + a > \frac{1}{4}
\end{cases}
\end{equation}
By (\ref{eq:21}),
\[d(A\omega_1 + B\omega_2)\wedge(A\omega_1 + B\omega_2) = 0\,.\]

\subsubsection{$a = b = 0$}
The structure equation (\ref{eq:21}) reduces to 
\begin{align*}
d\omega_1& = 0\, ,\\
d\omega_2& = (\omega_1 - \omega_2)\wedge\bar{\omega}_1\, .
\end{align*}
Since $d\omega_1 = 0$, by the complex Poincar\'e Lemma, there exists a holomorphic function $z$ such that $\omega_1 = dz$. So $d\omega_2 = (dz - \omega_2)\wedge d\bar{z}$, that is equivalent to $d(dz - \omega_2) = -(dz - \omega_2)\wedge d\bar{z}$. By the complex Frobenius Theorem, there exists a function $f$ and a holomorphic function $w$ such that $dz - \omega_2 = f\ dw$ and $df\wedge dw = fd\bar{z}\wedge dw$. By modifying $w$, we can write $dz - \omega_2 = e^{\bar{z}}\ dw$. So in the local coordinate system $(z,\ w)$, the coframing is
\begin{align*}
\omega_1& = dz\, ,\\
\omega_2& = - e^{\bar{z}}\ dw + dz\, .
\end{align*}
Take a discrete symmetry group of the coframing
\[\Gamma = \{ (2\alpha_0\pi i, \beta_0 + i\beta_1) | \alpha_0,\ \beta_0,\ \beta_1 \in \mathbb{Z}\} \cong \mathbb{Z}^3. \]

\noindent $\Gamma$ acts on the local coordinate as $(z,\ w)\mapsto (z + 2\alpha_0\pi i, \ w + \beta_0 + i\beta_1)$. This action keeps the coframing invariant. The quotient can be taken as
\[\mathbb{C}^2/\Gamma \cong \mathbb{R}^1\times S^1\times T^2\, .\]


%
%
\subsubsection{  $b^2 + a < \frac{1}{4}$ or $b^2 + a > \frac{1}{4}$}
Since the calculations are similar for these two cases, without loss of generality, assume $b^2 + a < \tfrac{1}{4}$. Define $U = -\frac{1}{2} + c +ib , V = -\frac{1}{2} - c +ib$ and $\theta_1 = \omega_1 + U\omega_2$, $\theta_2 = \omega_1 + V\omega_2$.

%
Then the structure equation is
\begin{align*}
d\theta_1 &= (\frac{1}{2} - c + ib)\bar{\theta}_2\wedge\theta_1\, ,\\
d\theta_2 &= (\frac{1}{2} + c + ib)\bar{\theta}_1\wedge\theta_2\, .
\end{align*}

By the complex Frobenius Theorem, there exist coordinates $z$ and $w$ and functions $f$ and $g$, such that $ \theta_1 = fdz,\ \theta_2 = \bar{g}d\bar{w}$.
Since $d(\theta_1\wedge\bar{\theta}_2) = d(fgdz\wedge dw)= 0$, so there exists a function $F(z,w)$ such that $fg = F(z,w)$. Since
\begin{align*}
d\theta_1 = df\wedge dz &= \Big(\frac{1}{2} - c + ib\Big)F(z,w) dw\wedge dz\, ,\\
d\bar{\theta}_2 = dg\wedge dw  &= \Big(\frac{1}{2} + c + ib\Big)F(z,w) dz\wedge dw\, ,
\end{align*}
$f$ and $g$ are both functions of $z, w$ and 
\begin{align*}
f_w &= \Big(\frac{1}{2} - c + ib\Big) fg\, ,\\
g_z &= \Big(\frac{1}{2} + c + ib\Big) fg\, .
\end{align*}
Thus
\begin{align*}
(\log fg)_{wz} &= \Big(\frac{f_w}{f}\Big)_z +\Big(\frac{g_w}{g}\Big)_z\\
&= \Big(\frac{1}{2} - c + ib\Big)g_z + \Big(\frac{g_w}{g}\Big)_z\\
&=\Big((\frac{1}{2} + ib)^2 - c^2 \Big)fg + \Big(\frac{g_w}{g}\Big)_z\, ,
\end{align*}
while
\begin{align*}
(\log fg)_{zw} &= \Big(\frac{f_z}{f}\Big)_w +\Big(\frac{g_z}{g}\Big)_w\\
&=  \Big(\frac{f_z}{f}\Big)_w + \Big(\frac{1}{2} + c + ib\Big) f_w\\
&=\Big(\frac{f_z}{f}\Big)_w + \Big((\frac{1}{2} + ib)^2 - c^2 \Big)fg.
\end{align*}
Since $(\log fg)_{wz} = (\log fg)_{zw}$, 
\[\Big(\frac{f_z}{f}\Big)_w =  \Big(\frac{g_w}{g}\Big)_z.\]
Since $(\log f)_{zw} = \Big(\frac{f_z}{f}\Big)_w $ and $(\log g)_{zw} = \Big(\frac{g_w}{g}\Big)_z$, 
\[(\log f)_{zw} = (\log g)_{zw}.\]
So there exist functions $A(z)$ and $B(w)$ such that
\[\frac{f(z,w)}{g(z,w)} = \frac{A(z)}{B(w)}.\]
So there exists a function $G(z, w)$ such that
\[ \theta_1 = A(z)G(z, w)dz,\ \ \ \theta_2 = \overline{B(w) G(z, w)}d\bar{w}.\]
After redefining $z$ and $w$ and the function $f(z,w)$, we can write
\[ \theta_1 = f(z, w)dz,\ \ \ \theta_2 = \overline{f(z, w)}d\bar{w}\]
\begin{equation*}
\begin{cases}
f_w &= \Big(\frac{1}{2} - c + ib\Big) f^2\\
f_z &= \Big(\frac{1}{2} + c + ib\Big) f^2
\end{cases}
\end{equation*}
Thus there exists a constant $C$ such that
\[f(z,w) = \frac{1}{-\Big(\frac{1}{2} + c + ib\Big) z + \Big(\frac{1}{2} - c + ib \Big)w + C}.\]
After translating $z$ or $ w$, 
\begin{equation*}
\begin{aligned}
\theta_1 &= \frac{dz}{-\Big(\frac{1}{2} + c + ib\Big) z + \Big(\frac{1}{2} - c + ib\Big)w},\\
\theta_2  &=  \frac{d\bar{w}}{-\Big(\frac{1}{2} + c + ib\Big) \bar{z} + \Big(\frac{1}{2} - c + ib\Big)\bar{w}}.
\end{aligned}
\end{equation*}
Now change the notation from $w$ to $\bar{w}$, then
\begin{equation*}
\begin{aligned}
\theta_1 &= \frac{dz}{-\Big(\frac{1}{2} + c + ib\Big) z + \Big(\frac{1}{2} - c + ib\Big)\bar{w}},\\
\theta_2  &=  \frac{dw}{-\Big(\frac{1}{2} + c + ib\Big) \bar{z} + \Big(\frac{1}{2} - c + ib\Big)w}.
\end{aligned}
\end{equation*}
So the Engel structure can be defined on the complex 2-plane except two lines
\begin{equation}
\begin{aligned}
-\Big(\frac{1}{2} + c + ib\Big) z + \Big(\frac{1}{2} - c + ib\Big)\bar{w} = 0,\\
-\Big(\frac{1}{2} + c + ib\Big) \bar{z} + \Big(\frac{1}{2} - c + ib\Big)w = 0.
\end{aligned}
\end{equation}

Now we can get the local coordinate representation of $\omega_1$ and $\omega_2$
\begin{equation*}
\begin{aligned}
\omega_1 &= \frac{1}{2c}\Big[(\frac{1}{2} + c + ib)\theta_1 + (-\frac{1}{2} + c + ib)\theta_2 \Big],\\
\omega_2  &=  \frac{1}{2c}(\theta_1 - \theta_2).
\end{aligned}
\end{equation*}
Let $\lambda$ and $\mu$ be two complex constants. The symmetry group of the coframing is
\[(z, w)\rightarrow (\lambda z,\  \bar{\lambda} w)\]
and 
\[(z, w)\rightarrow \left(z + \left(\frac{1}{2} - c + ib\right)\mu,\  w + \left(\frac{1}{2} + c + ib\right)\bar{\mu}\right).\]

\subsubsection{  $b^2 + a = \frac{1}{4}$}
The structure equation is
\begin{equation*}
\begin{aligned}
d\omega_1& = 2ib\bar{\omega}_1\wedge \omega_1 - \left(\frac{1}{4} - b^2 +ib\right)\bar{\omega}_2\wedge\omega_1 + \left(\frac{1}{4} - b^2-ib\right)\bar{\omega}_1\wedge\omega_2 \, ,\\
d\omega_2& = (\omega_1 - \omega_2)\wedge\bar{\omega}_1  - \left(\frac{1}{4} - b^2 + ib\right)\bar{\omega}_2\wedge\omega_2\, .
\end{aligned}
\end{equation*}
Define $\theta = \omega_1 + \left(-\tfrac{1}{2} +ib\right)\omega_2$. Then
\[d\theta = \Big(\frac{1}{2} +ib\Big)\bar{\theta}\wedge\theta.\]
Let $\theta = \alpha + i\beta$ be the real part and imaginary part decomposition. Then
\begin{align*}
d\alpha &= -2b\alpha\wedge\beta\, ,\\
d\beta &= \alpha\wedge\beta, 
\end{align*}
so $d(\alpha + 2b\beta) = 0$. Thus there exists a real function $x$ such that $\alpha + 2b\beta = dx$. Thus, $d\beta = dx\wedge \beta$. This is equivalent to $d(e^{-x}\beta) = 0$, that implies the existence of a real function $y$ such that $\beta = e^x dy$. So
\[\theta = dx - 2b e^x dy + i e^x dy.\]
Let $z = -e^{-x} - 2b y + iy$. Then
\[\theta = \frac{dz}{(-\frac{1}{2} +ib) z + (-\frac{1}{2} -ib)\bar{z}}.\]
Recall that the structure equation is 
\begin{equation*}
\begin{aligned}
d\omega_1 &= \left(\frac{1}{2} +ib\right)\bar{\theta}\wedge\omega_1 + \left(-\frac{1}{2} +ib\right)\bar{\omega}_1\wedge\theta,\\
d\omega_2 &= \theta\wedge\bar{\theta} + \left(\frac{1}{2} +ib\right)\left(\theta\wedge\bar{\omega}_2 + \bar{\theta}\wedge\omega_2\right).
\end{aligned}
\end{equation*}
Let $\omega_1 = \gamma + i\delta$ be the real part and imaginary part decomposition. Since $\theta = \alpha + i\beta$, the exterior derivative of $\omega_1$ can be written as

\begin{equation*}
\begin{aligned}
d\gamma &= \alpha\wedge\gamma + \beta\wedge\delta -2b\alpha\wedge\delta + 2b\beta\wedge\gamma,\\
d\delta &= 0.
\end{aligned}
\end{equation*}
So there exists a real function $u$ such that $\delta = du$. Since $\langle\gamma, du\rangle$ forms a Frobenius system, there exist real functions $p,\ q,\ v$ such that $\gamma = pdv + q du$. Write
\begin{equation*}
\begin{aligned}
dp &= p_x dx + p_y dy + p_u du + p_v dv\, ,\\
dq &= q_x dx + q_y dy + q_u du + q_v dv\, .
\end{aligned}
\end{equation*}
Then from the structure equation, the functions $p$ and $q$ satisfy
\begin{equation*}
\begin{aligned}
p_x &= \frac{p}{-x - 2by},\\
p_y &= \frac{2bp}{-x - 2by},\\
p_u &= q_v,\\
q_x &= \frac{q - 2b}{-x - 2by},\\
q_y &= \frac{1 + 2b q}{-x - 2by}.
\end{aligned}
\end{equation*}
From the first two equations, we know that there exists a function $C_1(u,v)$ such that $p = \frac{C_1(u, v)}{x + 2by}$. From the last two equations, we know that there exists a function $C_2(u,v)$ such that $q = \frac{(2bx - y) + C_2(u, v)}{x + 2by}$. From the third equation we know that $\frac{\partial C_1}{\partial u} = \frac{\partial C_2}{\partial v}$. Thus $d\Big(\int C_1(u, v)dv\Big) = C_1(u, v)dv + C_2(u, v) du$. Thus
\begin{align*}
\gamma &= \frac{C_1(u, v)}{x + 2by} dv +\frac{(2bx - y) + C_2(u, v)}{x + 2by} du\\
       &= \frac{C_1(u, v)dv + C_2(u, v) du}{x + 2by}  + \frac{2bx - y}{x + 2by} du\\
       &= \frac{d(\int C_1(u, v)dv)}{x + 2by}  + \frac{2bx - y}{x + 2by} du.
\end{align*}
Now define $(\int C_1(u, v)dv)$ as new $v$, then
\[\gamma = \frac{dv}{x + 2by}  + \frac{2bx - y}{x + 2by} du.\]
So
\begin{align*}
\omega_1 &= \gamma + i\delta\\
         &= \frac{dv}{x + 2by}  + \frac{2bx - y}{x + 2by} du + idu\\
         &= \frac{1}{x + 2by}\left[dv + (2b + i)(x + iy) du\right]\\
         &= \frac{1}{x + 2by}\left[d(v + (2b + i)(x + iy)u) -(2b + i)u\ d(x + iy)\right].
\end{align*}
Define $w = v + (2b + i)(x + iy)u$ and redefine $z = x + iy$. The coframing can be written as
\begin{align*}
\omega_1 &= \frac{1}{x + 2by}[dw - (2b + i)u\ dz]\\
         &= -\frac{1}{(-\frac{1}{2} +ib) z + (-\frac{1}{2} -ib)\bar{z}}\\
         &\quad \times\left[dw + \left(\frac{1}{2} - ib\right) \frac{w - \bar{w}}{(-\frac{1}{2} +ib) z + (-\frac{1}{2} -ib)\bar{z}} dz\right],\\
\omega_2 &= \frac{1}{-\frac{1}{2} + ib}(\theta - \omega_1)\\
        &= \frac{1}{(-\frac{1}{2} + ib) (-\frac{1}{2} +ib) z + (-\frac{1}{2} -ib)\bar{z}}\\
        &\quad \times\left[dz + dw + \left(\frac{1}{2} - ib\right) \frac{w - \bar{w}}{(-\frac{1}{2} +ib) z + (-\frac{1}{2} -ib)\bar{z}} dz\right],
\end{align*}
where $(z, w)$ is a local holomorphic coordinate system on the manifold. And $\omega_2 = 0$ defines the complex Engel structure.
Let $\lambda,\ s,\ t$ be any real constants. The coframing is invariant under the following local transformation
\[(z,\ w)\rightarrow\left(\lambda z + \left(b - \frac{1}{2}i\right)t,\ \lambda w + s \right).\]
We can take a discrete subgroup $\Gamma$ such that $M$ is locally biholomorphic to $\mathbb{C}^2/\Gamma \cong \mathbb{R}^2\times T^2$.
\subsection{Homogeneous Case C5}

Assume $a\neq 0$. Otherwise, this is a special case of homogeneous case C4, with $a + b^2 = \frac{1}{4}$. The structure equation is
\begin{equation*}
\begin{aligned}
d\omega_1 &= \left(b^2 - \frac{1}{4} -ib\right)\bar{\omega}_2\wedge\omega_1 +\left(\frac{1}{4}a + ab^2 +i\left(\frac{1}{2}ab + 2ab^3\right)\right)\bar{\omega}_2\wedge\omega_2 + 2ib\bar{\omega}_1\wedge\omega_1\\
 &\ \ \  + \left(\frac{1}{4} -\frac{1}{2}a - 2ab^2 - b^2 -ib\right)\bar{\omega}_1\wedge\omega_2 + \left(\frac{1}{2}a -2ab^2 -2iab\right)\omega_2\wedge\omega_1,\\
d\omega_2 &= \left(-\frac{1}{4} + \frac{1}{2}a -2ab^2 + b^2 +i \left(2ab - b\right)\right)\bar{\omega}_2\wedge\omega_2 + \omega_1\wedge\bar{\omega}_1 \\
&\ \ \ \ + \left(-a + 2iab\right)\omega_1\wedge\omega_2 + \left(-1 + a + 2iab\right)\omega_2\wedge\bar{\omega}_1.
\end{aligned}
\end{equation*}
Define $\theta = \omega_1 + \left(-\frac{1}{2} +ib\right)\omega_2$, that satisfies
\begin{equation}\label{eq:88}
d\theta = \left(\frac{1}{2} +ib\right)\bar{\theta}\wedge\theta.
\end{equation}
This structure equation is same as that of the case C4, with $b^2 + a = \frac{1}{4}$. But in case C5, $a$ and $b$ do not have to satisfy this relation.

Let $\theta = \alpha + i\beta$ be the real and imaginary part decomposition. By (\ref{eq:88}),
\begin{equation*}
\begin{aligned}
d\alpha &= -2b\alpha\wedge\beta,\\
d\beta &= \alpha\wedge\beta.
\end{aligned}
\end{equation*}
Since $d(\alpha + 2b\beta) = 0$, there exists a real function $x$ such that $\alpha + 2b\beta = dx$. Thus $d\beta = dx\wedge \beta$. So there exists a real function $y$ such that $\beta = e^x dy$, that yields
\[\theta = dx - 2b e^x dy + i e^x dy\, .\]
Let $z = -e^{-x} - 2b y + iy$. Then
\[\theta = \frac{dz}{(-\frac{1}{2} +ib) z + (-\frac{1}{2} -ib)\bar{z}} \, .\]
The symmetry groups of the coframing are different for different parameters $a$ and $b$. In the following sections, we will consider the following cases:
\begin{enumerate}
\item $a \neq \pm\frac{1}{2}$
\item $a = \frac{1}{2}, b = 0$
\item $a = \frac{1}{2}, b \neq 0$
\item $a = -\frac{1}{2}$
\end{enumerate}

\subsubsection{ $a \neq \pm\frac{1}{2}$}
Define $A = r + i s$, where
\begin{equation*}
\begin{aligned}
r &= \frac{1}{\Big(\frac{1}{2} + a\Big) (1 + 4b^2)}\\
s &= \frac{2b}{\Big(\frac{1}{2} - a\Big) (1 + 4b^2)}.\\
\end{aligned}
\end{equation*}
\noindent After defining $\theta_2 = \omega_2 + A\theta$, 
the structure equation reduces to
\[d\theta_2 = 2a\Big(-\frac{1}{2} + ib\Big)\theta\wedge\theta_2 + (2a - 1)\Big(\frac{1}{2} + ib\Big)\theta_2\wedge\bar{\theta} + \Big(\frac{1}{2} + ib\Big)\theta\wedge\bar{\theta}_2.\]
Let $\theta_2 = \gamma + i\delta$ and
$\theta = \frac{dx + i dy}{A(z)}$, where 
$A(z) = (-\frac{1}{2} +ib) z + (-\frac{1}{2} -ib)\bar{z} = -x -2by$ is a real function. Then
\begin{equation}\label{eq:89}
\begin{aligned}
d\gamma &= \Big(-\frac{2a}{A(z)} + \frac{1}{A(z)}\Big)dx\wedge\gamma + \frac{1}{A(z)} dy\wedge\delta -\frac{4ab}{A(z)}dy\wedge\gamma\\
d\delta &= \frac{2b}{A(z)}dx\wedge\gamma + \Big(-\frac{4ab}{A(z)} + \frac{2b}{A(z)}\Big)dy\wedge\delta -\frac{2a}{A(z)}dx\wedge\delta.
\end{aligned}
\end{equation}
By (\ref{eq:89}), $\langle\gamma, dy\rangle$ and $\langle\delta, dx\rangle$ are two Frobenius systems, that implies the existence of functions $p, \ q, \ r,\ s$ and $u,\ v$ such that
\begin{equation*}
\begin{aligned}
\gamma &= p du + q dy,\\
\delta &= r dv + sdx.\\
\end{aligned}
\end{equation*}
Write
\begin{equation*}
\begin{aligned}
dp &= p_x dx + p_y dy + p_u du + p_v dv,\\
dq &= q_x dx + q_y dy + q_u du + q_v dv,\\
dr &= r_x dx + r_y dy + r_u du + r_v dv,\\
ds &= s_x dx + s_y dy + s_u du + s_v dv.
\end{aligned}
\end{equation*}
By (\ref{eq:89}),
\begin{align*}
d\gamma &= dp\wedge du + dq\wedge dy\\
        &= p_x dx\wedge du + p_y dy\wedge du + p_v dv\wedge du + q_x dx\wedge dy + q_u du\wedge dy + q_v dv\wedge dy\\
        &= \frac{1}{-x -2by}[p(1 - 2a)dx\wedge du + (q(1 - 2a) - s)dx\wedge dy\\
        &\ \ \  + r dy\wedge dv - 4ab p dy\wedge du].
\end{align*}
Then the functions $p$ and $q$ must satisfy
\begin{equation*}
\begin{aligned}
p_y -q_u &= \frac{-4ab p}{-x -2by},\\
p_x &= \frac{1 - 2a}{-x -2by} p,\\
p_v &= 0,\\
q_v &= -\frac{r}{-x -2by},\\
q_x &= \frac{q(1 - 2a) - s}{-x -2by}.
\end{aligned}
\end{equation*}
Thus there exists a function $C_1(u,y)$ such that $p = (x + 2by)^{2a - 1}C_1(u, y)$. After redefining $u$ and $q$, we can assume $C_1(u,y) = 1$. Then
\[p = (x + 2by)^{2a - 1}.\]
By (\ref{eq:89}),
\begin{align*}
d\delta &= dr\wedge dv + ds\wedge dx\\
        &=  r_x dx\wedge dv + r_y dy\wedge dv + r_u du\wedge dv + s_y dy\wedge dx + s_u du\wedge dx + s_v dv\wedge dx\\
        &= \frac{1}{-x -2by}\Big[2bp dx\wedge du + (2bq + (4ab - 2b)s)dx\wedge dy\\
        &\ \ \  + (-4ab + 2b)rdy\wedge dv - 2ar dx\wedge dv \Big].
\end{align*}
Then the functions $r$ and $s$ must satisfy
\begin{equation}\label{eq:90}
\begin{aligned}
r_x - s_v &= -\frac{2ar}{-x -2by},\\
r_y &= \frac{-4ab + 2b}{-x -2by} r,\\
r_u &= 0,\\
s_y &= \frac{-2bq + (-4ab + 2b)s}{-x -2by},\\
s_u &= -\frac{2bp}{-x -2by}.
\end{aligned}
\end{equation}
So there exists a function $C_2(x,v)$ such that $r = (x + 2by)^{2a - 1}C_2(x,v)$. After redefining $v$ and $s$, we can assume $C_2(x,v) = 1$. Then
\[r = (x + 2by)^{2a - 1}.\]
Substituting this equation into (\ref{eq:90}), the equations are as follows:
\begin{equation*}
\begin{aligned}
q_u &= -2b(x + 2by)^{2a - 2},\\
q_v &= (x + 2by)^{2a - 2},\\
q_x &= \frac{(2a - 1) q + s}{x + 2by},\\
s_v &= -(x + 2by)^{2a - 2},\\
s_u &= 2b (x + 2by)^{2a - 2},\\
s_y &= \frac{2b q + (4ab - 2b)s}{x + 2by}.
\end{aligned}
\end{equation*}
This yields
\begin{equation}\label{eq:91}
\begin{aligned}
\theta_2 &= \gamma + i\delta\\
         &= (x + 2by)^{2a - 1}(du + i dv) + i s\left(dx -i\frac{q}{s}dy\right).\\
\end{aligned}
\end{equation}
To get a local holomorphic coordinate system, let $q = - s$. Then
\begin{equation*}
\begin{aligned}
q_u &= -2b(x + 2by)^{2a - 2},\\
q_v &= (x + 2by)^{2a - 2},\\
q_x &= \frac{(2a - 2) q}{x + 2by},\\
q_y &= -\frac{(-4ab + 4b)q}{x + 2by}.
\end{aligned}
\end{equation*}
From these equations, we get
\[q = (1 - 2bu + v)(x + 2by)^{2a - 2}\]
Substituting this equation into (\ref{eq:91}) yields
\begin{align*}
\theta_2 &= \gamma + i\delta\\
         &= (p du + q dy) + i(r dv + sdx)\\
         &= ((x + 2by)^{2a - 1} du + (1 - 2bu + v)(x + 2by)^{2a - 2} dy) \\
         &\ \ \ \ + i((x + 2by)^{2a - 1} dv - (1 - 2bu + v)(x + 2by)^{2a - 2} dx)\\
         &= (x + 2by)^{2a - 1} d(u + iv) -i(1 - 2bu + v)(x + 2by)^{2a - 2} d(x + iy).
\end{align*}
Define $w = u + iv$. We get the local coordinate representation of $\theta_2$:
\[\theta_2 =(-A(z))^{2a - 1}\left[dw  +\left[i +\left(\frac{1}{2} -bi\right)w -\left(\frac{1}{2} + bi\right)\bar{w}\right]\frac{dz}{A(z)}\right]. \]
Recall that
\[\theta = \omega_1 + \left(-\frac{1}{2} +ib\right)\omega_2,\]
\[\theta_2 = \omega_2 + \left[\frac{1}{\left(\frac{1}{2} + a\right) \left(1 + 4b^2\right)} + i\frac{2b}{\left(\frac{1}{2} - a\right) \left(1 + 4b^2\right)}\right]\theta.\]
Thus in the local holomorphic coordinate system $(z,\ w)$, the coframing can be written as
\begin{align*}
\omega_1 &=\left\{ 1 + \left(-\frac{1}{2} +ib\right)\left[ \frac{1}{\left(\frac{1}{2} + a\right) (1 + 4b^2)} + i\frac{2b}{\left(\frac{1}{2} - a\right) (1 + 4b^2)}\right]\right\}\theta - \left(-\frac{1}{2} +ib\right)\theta_2\\
\omega_2 &= \theta_2 -  \left[\frac{1}{\left(\frac{1}{2} + a\right) (1 + 4b^2)} + i\frac{2b}{\left(\frac{1}{2} - a\right) (1 + 4b^2)}\right]\theta.
\end{align*}

\subsubsection{ $a =\frac{1}{2},\ b =0$}
In this case, we can choose $(a, b)$ such that 
\[\lim_{\substack{a \to\frac{1}{2} \\ b \to 0}}\frac{2b}{\Big(\frac{1}{2} - a\Big) (1 + 4b^2)} = 0\]
Then all the formula in the case $a \neq\tfrac{1}{2}$ applies to this case $a =\tfrac{1}{2},\ b =0$. The local coordinate representation of the coframing is
\begin{align*}
\omega_1 &=\left(\frac{1}{2} +ib\right)\theta - \left(-\frac{1}{2} +ib\right)\theta_2\\
\omega_2 &= \theta_2 - \theta.
\end{align*}

\subsubsection{ $a =\frac{1}{2},\ b \neq 0$}
The structure equation is
\begin{equation}\label{eq:92}
d\omega_2 = \omega_1\wedge\bar{\omega}_1 + \left(-\frac{1}{2} +ib\right)\omega_1\wedge\omega_2 + \left(-\frac{1}{2} +ib\right)\omega_2\wedge\bar{\omega}_1.
\end{equation}
Substituting $\theta = \omega_1 +(-\frac{1}{2} +ib)\omega_2$ into (\ref{eq:92}) yields
\begin{align*}
d\omega_2 &= \Big[\theta - \Big(-\frac{1}{2} +ib\Big)\omega_2 \Big]\wedge\Big[\bar{\theta} - \Big(-\frac{1}{2} - ib\Big)\bar{\omega}_2 \Big]\\
         &\ \ \ \  + \Big(\frac{1}{2} +ib\Big)\theta\wedge\omega_2 + \Big(-\frac{1}{2} +ib\Big)\omega_2\wedge\Big[\bar{\theta} - \Big(-\frac{1}{2} - ib\Big)\bar{\omega}_2 \Big] \\
          &= \theta\wedge\bar{\theta} + \Big(\frac{1}{2} +ib\Big)\theta\wedge\bar{\omega}_2 + \Big(-\frac{1}{2} +ib\Big)\theta\wedge\omega_2.
\end{align*}
So $\langle\omega_2, \theta\rangle$ is a Frobenius system.
Write $\omega_2 = \gamma + i\delta$. Then
\[d(\gamma + i\delta) = \frac{-2i}{A(z)^2}dx\wedge dy + \frac{i}{A(z)}(dx + i dy)\wedge(-\delta + 2b\gamma).\]
Hence
\begin{equation*}
\begin{aligned}
d\gamma &= -\frac{1}{A(z)}dy\wedge(-\delta + 2b\gamma)\\
d\delta &= -\frac{2}{A(z)}dx\wedge dy + \frac{1}{A(z)}dx\wedge(-\delta + 2b\gamma)\, .
\end{aligned}
\end{equation*}
So $\langle\gamma,\ dy\rangle$ and $\langle\delta,\ dx\rangle$ are two Frobenius systems. There exist functions $p,\ q,\ r,\ s$ and $u,\ v$ such that
\begin{align*}
\gamma &= p du + q dy\\
\delta &= r dv + sdx.
\end{align*}
From the structure equation,
\begin{align*}
d\gamma &= dp\wedge du + dq\wedge dy\\
        &= p_x dx\wedge du + p_y dy\wedge du + p_v dv\wedge du + q_x dx\wedge dy + q_u du\wedge dy + q_v dv\wedge dy\\
        &= -\frac{1}{A(z)} dy\wedge(-rdv - sdx + 2bpdu).
\end{align*}
So
\begin{equation}\label{eq:93}
\begin{aligned}
p_y -q_u &= \frac{2b p}{x + 2by},\\
p_x &= 0,\\
p_v &= 0,\\
q_v &= \frac{r}{x + 2by},\\
q_x &= \frac{s}{x +2by}.
\end{aligned}
\end{equation}
Since there are ambiguities for choosing $p$ and $u$, by modifying $u$ we can arrange $p = 1$.

From the structure equation,
\begin{align*}
d\delta &= dr\wedge dv + ds\wedge dx\\
        &=  r_x dx\wedge dv + r_y dy\wedge dv + r_u du\wedge dv + s_y dy\wedge dx + s_u du\wedge dx + s_v dv\wedge dx\\
        &=-\frac{2}{A(z)^2} dx\wedge dy + \frac{1}{A(z)} dx\wedge(-r dv + 2bp du + 2bq dy).
\end{align*}
Then the functions $r$ and $s$ must satisfy
\begin{equation}\label{eq:94}
\begin{aligned}
r_x - s_v &= -\frac{r}{-x -2by},\\
r_y &= 0,\\
r_u &= 0,\\
s_y &= \frac{2}{(x + 2by)^2} + \frac{2b}{x + 2by}q,\\
s_u &= -\frac{2bp}{-x -2by}.
\end{aligned}
\end{equation}
Since there is ambiguity for choosing $r$ and $v$, by modifying $v$ we can arrange $r = 1$.

Thus  (\ref{eq:93}) and (\ref{eq:94}) reduces to

\begin{equation*}
\begin{aligned}
q_u &= -\frac{2b}{x + 2by},\\
q_v &= \frac{1}{x + 2by},\\
q_x &= \frac{s}{x +2by},\\
s_v &= -\frac{1}{x + 2by},\\
s_y &= \frac{2}{(x + 2by)^2} + \frac{2b}{x + 2by}q,\\
s_u &= \frac{2b}{x +2by}.
\end{aligned}
\end{equation*}
Thus there exist functions $C_1(x, y)$ and $C_2(x, y)$ such that
\[q = \frac{1}{A(z)}(2b u - v) + C_1(x, y)\]
and
\[s = -\frac{1}{A(z)}(2b u - v) + C_2(x, y).\]
The functions $C_1(x, y)$ and $C_2(x, y)$ satisfy
\begin{equation*}
\begin{aligned}
C_{1x} &= \frac{1}{x + 2b y}C_2,\\
C_{2y} &= \frac{2b}{x + 2b y} C_1 + \frac{2}{(x + 2by)^2}.
\end{aligned}
\end{equation*}
Choose $C_2(x, y) = - C_1(x, y)$ such that $z$ is a local holomorphic coordinate. Then
\begin{equation*}
\begin{aligned}
C_{1x} &= -\frac{1}{x + 2b y}C_1,\\
C_{1y} &= -\frac{2b}{x + 2b y} C_1 - \frac{2}{(x + 2by)^2}.
\end{aligned}
\end{equation*}
Then
\[C_1(x, y) = -\frac{\ln(x + 2by)}{b(x + 2by)}.\]
So
\begin{align*}
\omega_2 &= \gamma + i\delta\\
         &= (p du + q dy) +i( r dv + sdx)\\
         &= d(u + iv)+ \left( \frac{1}{A(z)}(2b u - v) -\frac{\ln(x + 2by)}{b(x + 2by)}\right) d(y- ix).
\end{align*}
Let $w = u + iv$. Then
\[\omega_2 = dw -i\frac{1}{A(z)}\left[\Big(b + \frac{i}{2}\Big)w + \Big(b - \frac{i}{2}\Big)\bar{w} + \frac{1}{b}\ln(-A(z)) \right]dz.\] 
Recall that
\[\theta = \omega_1 + \left(-\frac{1}{2} +ib\right)\omega_2.\]
So we can write $\omega_1$ as
\[\omega_1 = \frac{dz}{A(z)} - \left(-\frac{1}{2} +ib\right)\omega_2.\]

\subsubsection{$a = -\frac{1}{2}$}
The structure equation is
\begin{align*}
d\omega_2 &= \Big(-\frac{1}{2} + 2 b^2 -2ib\Big)\bar{\omega}_2\wedge\omega_2 + \omega_1\wedge\bar{\omega}_1 + \Big(\frac{1}{2} -ib\Big)\omega_1\wedge\omega_2 + \Big(-\frac{3}{2} - ib\Big)\omega_2\wedge\bar{\omega}_1\\
&= \theta\wedge\bar{\theta} +  \Big(\frac{1}{2} +ib\Big) \theta\wedge\bar{\omega}_2 -2 \Big(\frac{1}{2} +ib\Big) \omega_2\wedge\bar{\theta} +  \Big(\frac{1}{2} -ib\Big)\theta\wedge\omega_2.
\end{align*}
Let $\omega_2 = \gamma + i\delta$ be the real and imaginary part decomposition. Then
\begin{equation*}
\begin{aligned}
d\gamma &= \frac{2}{A(z)}dx\wedge \gamma + \frac{1}{A(z)}dy\wedge\delta + \frac{2b}{A(z)} dy\wedge \gamma,\\
d\delta &= -\frac{2}{A(z)^2}dx\wedge dy + \frac{1}{A(z)}dx\wedge\delta + \frac{2b}{A(z)}dx\wedge\gamma + \frac{4b}{A(z)}dy\wedge\delta.
\end{aligned}
\end{equation*}
So $\langle\gamma,\ dy\rangle$ and $\langle\delta,\ dx\rangle$ are two Frobenius systems. So there exist functions $p,\ q,\ r,\ s$ and $u,\ v$ such that
\begin{equation*}
\begin{aligned}
\gamma &= p du + q dy,\\
\delta &= r dv + sdx.
\end{aligned}
\end{equation*}
From the structure equation,
\begin{align*}
d\gamma &= dp\wedge du + dq\wedge dy\\
        &= p_x dx\wedge du + p_y dy\wedge du + p_v dv\wedge du + q_x dx\wedge dy + q_u du\wedge dy + q_v dv\wedge dy\\
        &=\frac{1}{A(z)}[2dx\wedge(pdu + qdy) + dy\wedge(rdv + sdx) + 2bdy\wedge(pdu + qdy)].\\
\end{align*}
This yields
\begin{equation}\label{eq:95}
\begin{aligned}
p_y -q_u &= \frac{2b p}{A(z)},\\
p_x &= \frac{2 p}{A(z)},\\
p_v &= 0,\\
q_v &= -\frac{r}{A(z)},\\
q_x &= \frac{2q - s}{A(z)}.
\end{aligned}
\end{equation}
Since there is ambiguity for choosing $p$ and $u$, by modifying $u$ we arrange $p = \frac{1}{A(z)^2}$.

From the structure equation,
\begin{align*}
d\delta &= dr\wedge dv + ds\wedge dx\\
        &=  r_x dx\wedge dv + r_y dy\wedge dv + r_u du\wedge dv + s_y dy\wedge dx + s_u du\wedge dx + s_v dv\wedge dx\\
        &=-\frac{2}{A(z)^2} dx\wedge dy\\
        & + \frac{1}{A(z)} [dx\wedge(rdv + sdx)+ 2b dx\wedge(pdu + qdy) + 4b dy\wedge(rdv + sdx)].
\end{align*}
The functions $r$ and $s$ must satisfy
\begin{equation}\label{eq:96}
\begin{aligned}
r_x - s_v &= \frac{r}{A(z)},\\
r_y &=  \frac{4br}{A(z)},\\
r_u &= 0,\\
s_y &= \frac{2}{A(z)^2} + \frac{-2bq + 4bs}{A(z)},\\
s_u &= -\frac{2bp}{A(z)}.
\end{aligned}
\end{equation}
Since there is ambiguity for choosing $r$ and $v$, by modifying $v$ we arrange $r = \frac{1}{A(z)^2}$.

(\ref{eq:95}) and (\ref{eq:96}) reduce to
\begin{equation*}
\begin{aligned}
q_u &= -\frac{2b}{(x + 2by)^3},\\
q_v &= \frac{1}{(x + 2by)^3},\\
q_x &= \frac{2q-s}{A(z)},\\
s_v &= -\frac{1}{(x + 2by)^3},\\
s_y &= \frac{2}{(x + 2by)^2} + \frac{-2bq + 4bs}{-(x + 2by)},\\
s_u &= \frac{2b}{(x +2by)^3}.
\end{aligned}
\end{equation*}
Then there are functions $C_1(x, y)$ and $C_2(x, y)$ such that
\begin{equation}\label{eq:96}
q = \frac{1}{A(z)^3}(2b u - v) + C_1(x, y)
\end{equation}
and
\begin{equation}\label{eq:97}
s = -\frac{1}{A(z)^3}(2b u - v) + C_2(x, y).
\end{equation}
The functions $C_1(x, y)$ and $C_2(x, y)$ satisfy
\begin{equation*}
\begin{aligned}
C_{1x} &= \frac{C_2- 2C_1}{x + 2b y},\\
C_{2y} &= \frac{2bC_1 - 4b C_2}{x + 2b y}  + \frac{2}{(x + 2by)^2}.
\end{aligned}
\end{equation*}
Choose $C_2(x, y) = - C_1(x, y)$ such that $z$ is a local holomorphic coordinate. Then
\begin{equation*}
\begin{aligned}
C_{1x} &= -\frac{3}{x + 2b y}C_1,\\
C_{1y} &= -\frac{6b}{x + 2b y} C_1 - \frac{2}{(x + 2by)^2}.
\end{aligned}
\end{equation*}
We can solve $C_1(x, y)$
\[C_1(x, y) = -\frac{2y(x + by)}{(x+ 2by)^3}.\]
By (\ref{eq:96}), we get
\[q = \frac{1}{A(z)^3}(2b u - v) -2\frac{y(x + by)}{(x+ 2by)^3}.\]
By (\ref{eq:97}), we get
\[s = -\frac{1}{A(z)^3}(2b u - v) +2\frac{y(x + by)}{(x+ 2by)^3}.\]
So
\begin{align*}
\omega_2 &= \gamma + i\delta\\
         &= (p du + q dy) +i( r dv + sdx)\\
         &= \frac{1}{A(z)^2} d(u + iv)+ \Big(\frac{1}{A(z)^3}(2b u - v) -2\frac{y(x + by)}{(x+ 2by)^3}\Big) d(y- ix).
\end{align*}
Let $w = u + iv$. Then
\[\omega_2 =\frac{1}{A(z)^2} dw -i\frac{1}{A(z)^3}\Big[\Big(b + \frac{i}{2}\Big)w + \Big(b - \frac{i}{2}\Big)\bar{w}  -\frac{i}{2}(z - \bar{z})((1-ib)z + (1 + ib)\bar{z}) \Big]dz.\] 
Recall that
\[\theta = \omega_1 +\left(-\frac{1}{2} +ib\right)\omega_2,\]
so we can write $\omega_1$ such that
\[\omega_1 = \frac{dz}{A(z)} - (-\frac{1}{2} +ib)\omega_2\]

\subsubsection{Compact case of type $C5$}
From the structure equation, 
\[d(\omega_1\wedge\bar{\omega}_2\wedge\omega_2) = (1 - 2a)(1 + 2b i)\omega_1\wedge\bar{\omega}_1\wedge\omega_2\wedge\bar{\omega}_2\neq 0.\]

\noindent If $a \neq \frac{1}{2}$, there is  no compact quotient of type $C5$. In the following analysis, we assume $a = \frac{1}{2}$. 

Recall that
\begin{equation}\label{eq:77}
\begin{aligned}
d\theta &= \Big(\frac{1}{2} +ib\Big)\bar{\theta}\wedge\theta,\\
d\omega_2 &=  \theta\wedge\bar{\theta} + \left(\frac{1}{2} +ib\right)\theta\wedge\bar{\omega}_2 + \left(-\frac{1}{2} +ib\right)\theta\wedge\omega_2.
\end{aligned}
\end{equation}

Let $\theta = \alpha + i\beta$ and $\omega_2 = \gamma + i\delta$ be the real and imaginary part decompositions. By (\ref{eq:77}), 
\begin{equation*}
\begin{aligned}
d\alpha &= -2b\alpha\wedge\beta,\\
d\beta &= \alpha\wedge\beta,\\
d\gamma &= \beta\wedge\delta - 2b\beta\wedge\gamma,\\
d\delta &= -2\alpha\wedge\beta - \alpha\wedge\delta + 2b\alpha\wedge\gamma.
\end{aligned}
\end{equation*}
Let $e_1, e_2, e_3, e_4$ be left-invariant vector fields dual  to the left-invariant forms $\alpha, \beta, \gamma, \delta$, respectively. The nontrivial brackets are
\begin{align*}
[e_1, e_2] &= -2be_1 + e_2 -2e_4,\\
[e_1, e_3] &= 2be_4,\\
[e_1, e_4] &= -e_4,\\
[e_2, e_3] &= -2be_3,\\
[e_2, e_4] &= e_3.
\end{align*}
Make a basis change: $\tilde e_3 = e_3 + 2be_4, \tilde e_2 = e_2 -2be_1$. Then
\begin{align*}
[e_1, \tilde e_3] &= 0,\\
[e_2, \tilde e_3] &= 0.\\
\end{align*}
After dropping the tildes, the nontrivial brackets are
\begin{align*}
[e_1, e_2] &= e_2 - 2e_4,\\
[e_1, e_4] &= -e_4,\\
[e_2, e_4] &= e_3.
\end{align*}
Define $\tilde e_2 = e_2 - e_4$. Then
\begin{align*}
[e_1, \tilde e_2] &= \tilde e_2,\\
[e_1, e_4] &= -e_4,\\
[\tilde e_2, e_4] &= e_3.
\end{align*}
Then define $X_1 = -e_3, X_2 = e_4, X_3 = \tilde e_2, X_4 = e_1$. The nontrivial brackets are
\begin{align*}
[X_2, X_3] &= X_1,\\
[X_2, X_4] &= X_2,\\
[X_3, X_4] &= -X_3.
\end{align*}
By the classification result of solvmanifolds \cite{MR3480018}, there exists a co-compact lattice $\Gamma$ such that $G/\Gamma$ is compact and supports a homogeneous complex Engel structure.

\subsection{Homogeneous Case C6}
In this case, the symmetry groups of the coframing are different for different parameters $a$ and $b$. We will consider the following 2 cases:
\begin{enumerate}
\item $a = -\frac{\pi}{2} + 2k\pi,\ k\in \mathbb{Z}$.

In this case, the constants $(p_1, p_2, q_1, q_2, r_1, r_2)$ are 
 $(\frac{1}{2} - ib,\ 0,\ 2bi,\ \frac{1}{2}ib(2b + i)^2,\ -b(-2b + i),\ -\frac{1}{4}b(-2b + i)(2b + i)^2)$. This is a special case of homogeneous case C3, with $a= b^2$. There is no compact quotient that can support a homogeneous complex Engel structure unless $b = 0$. Under this condition, this is a special case of case $C1$.

 \item $a \neq -\frac{\pi}{2} + 2k\pi,\ k\in \mathbb{Z}$.

By the structure equation, we get
\begin{align*}
d(\bar{\omega}_1\wedge\omega_2\wedge\bar\omega_2) &= -2\left(\left(b + \tfrac{1}{2}i\right)\cos{a} + \left(-\tfrac{1}{2} + bi\right)\left( \sin{a} + 1\right)\right)\omega_1\wedge\bar{\omega}_1\wedge\omega_2\wedge\bar\omega_2,\\
d(\omega_1\wedge\bar\omega_1\wedge\bar\omega_2) &= -\tfrac{1}{2}\left(-\sin{a} + 2b\cos{a} - 1\right)\left(4b^2 - 1 + 4bi\right)\omega_1\wedge\bar{\omega}_1\wedge\omega_2\wedge\bar\omega_2.
\end{align*}
It is easy to verify that $d(\bar{\omega}_1\wedge\omega_2\wedge\bar\omega_2) = d(\omega_1\wedge\bar\omega_1\wedge\bar\omega_2) = 0$ if and only if $a = -\frac{\pi}{2} + 2k\pi,\ k\in \mathbb{Z}$. But this is contradictory to our assumption. Thus the volume form is exact in this case. By Stokes' Theorem, there does not exist compact quotient that supports a homogeneous complex Engel structure. 

\end{enumerate}

\end{document}